\documentclass[10pt]{article}
\usepackage[a4paper, total={6in, 8in}]{geometry}
\usepackage{amssymb}
\usepackage{amsmath}
\usepackage{amsthm}
\usepackage{mathdots} 
\usepackage{graphicx}
\usepackage{multirow}

\usepackage{lipsum}
\usepackage{amsfonts}
\usepackage{epstopdf}
\usepackage{algorithmic}
\usepackage{algorithm}
\ifpdf
  \DeclareGraphicsExtensions{.eps,.pdf,.png,.jpg}
\else
  \DeclareGraphicsExtensions{.eps}
\fi

\allowdisplaybreaks

\newcommand{\diagi}{\begin{smallmatrix}\vspace{-0.5ex}\textrm{\normalsize diag}\\\vspace{-0.8ex}i\in\mathcal{I}_{n}\end{smallmatrix}}

\usepackage{todonotes}
\usepackage{ulem}
\usepackage{xcolor}

\newtheorem{theorem}{Theorem}
\newtheorem{lemma}{Lemma}

\newtheorem{remark}{Remark}
\usepackage{amsopn}
\DeclareMathOperator{\diag}{diag}

\title{Analysis on aggregation and block smoothers in multigrid methods for block Toeplitz linear systems
}

\author{Matthias Bolten, Marco Donatelli, Paola Ferrari, Isabella Furci}

\begin{document}

\maketitle

\begin{abstract}

We present novel improvements in the context of symbol-based multigrid procedures for solving large block structured linear systems. We study the application of an aggregation-based grid transfer operator that transforms the symbol of a block Toeplitz matrix from matrix-valued to scalar-valued at the coarser level.
Our convergence analysis of the Two-Grid Method (TGM)
reveals the connection between the features of the scalar-valued symbol at the coarser level and the properties of the original matrix-valued one.
 This allows us to prove the convergence of a V-cycle multigrid with
 standard grid transfer operators for scalar Toeplitz systems at the coarser levels.
Consequently, we extend the class of suitable smoothers for block Toeplitz matrices, focusing on the efficiency of block strategies, particularly the relaxed block Jacobi method. General conditions on smoothing parameters are derived, with emphasis on practical applications where these parameters can be calculated with negligible computational cost. 
 We test the proposed strategies on linear systems stemming from the discretization of differential problems with  $\mathbb{Q}_{d} $ Lagrangian FEM or B-spline with non-maximal regularity. The numerical results 
 show in both cases computational advantages compared to existing methods for block structured linear systems. 
\end{abstract}

65N55 (multigrid methods), 65F08   (preconditioners for iterative methods), 34L20  (eigevalue distributions), 15B05  (Toeplitz matrices)

\section{Introduction}
\label{sec:intro}

In this paper, we propose a novel symbol-based multigrid procedure designed to tackle the challenge of solving large block Toeplitz linear systems, where the coefficient matrix entries are generic small matrices instead of scalars.
Linear systems with multilevel block-Toeplitz coefficient matrices
arise in the discretization of many differential equations, like the
${\mathbb{Q}_r}$ Lagrangian finite element method (FEM)
 or B-Spline approximation of second order differential problems \cite{qp,MR4019290}.

Efficient symbol-based multigrid strategies for the solution of such linear systems have been proposed and studied in \cite{MR4389580,MR4284081}, with grid transfer operators that preserve the block structure and features of the original matrix in the coarser spaces. Here, the selection of multigrid parameters and the convergence analysis are strictly related to the properties of the matrix-valued generating functions of the block Toeplitz matrices. Furthermore, in \cite{An2024} it is shown that effective solution strategies can also be based on grid transfer operators that aggregate the unknowns and transform the block problem into a scalar one at the coarser level, with significant computational advantages. In particular, the authors prove the convergence of the Two-Grid Method (TGM) and suggest a V-cycle strategy where the properties of the scalar system at the coarser level are crucial for the convergence of the method. General convergence results for the TGM method have been established in \cite{MR2114296,MR2150164} and references therein. Moreover, automatic algorithmic proposals can be developed exploiting a posteriori evaluation of the quality of the aggregation procedure \cite{MR3264812,MR4615364}. 

The first goal of the present paper is to analyze the symbol at the coarser level to prove the convergence of the V-cycle. This analysis reveals that the features of the scalar-valued symbol at the coarser level are closely linked to the properties of the original matrix-valued one, which is crucial for two main reasons.  First, it enables the straightforward derivation of convergence conditions for the V-cycle used in \cite{An2024}. Secondly, this analysis also improves the algorithmic proposal in \cite{MR4389580,MR4284081} by simplifying the calculations to the diagonalization of the matrix-valued symbol at a singular point. Consequently, by concentrating on scalar-valued functions, we facilitate the evaluation of approximating and smoothing properties, making them more accessible for computational implementation.

Algebraic multigrid method based on smoothed aggregation have proved to be very efficient methods for the solution of symmetric, positive definite systems arising from finite element
discretization of elliptic boundary value problems \cite{BDH,MR2179900}.
Another goal of this paper is to expand the class of smoothers for block Toeplitz matrices by proving the efficiency of block strategies more suitable than the scalar ones discussed in \cite{MR4389580,MR4284081}. In particular, we focus on the relaxed block Jacobi method and we derive general conditions on the smoothing parameter for a general convergence result. Moreover, we show that in many practical applications, the smoothing parameter can be straightforwardly computed from the generating functions.

Despite the combination with stronger smoothing techniques, aggregation-based restriction strategies still lack effectiveness, as mentioned in \cite{BDH,MR2385900}. However, as Braess demonstrates in \cite{MR1370108}, the performance of these multigrid methods, particularly for second-order elliptic problems, can be significantly improved. This improvement comes from carefully adjusting the coarse grid correction, a strategy known as over-relaxation. Selecting the ideal over-relaxation parameter is non-trivial. Yet, for block Toeplitz linear systems, we introduce a method that simplifies this selection by only requiring calculations with the symbol, thereby making it applicable in real-world scenarios. To demonstrate the effectiveness of our approach, we carry out comprehensive numerical experiments, comparing the performance of our over-relaxed aggregation-based multigrid strategy with current techniques for block Toeplitz linear systems.

The paper is organized as follows. In Section \ref{sec:two_grid} we report useful notation and preliminary results on multigrid methods and block circulant and Toeplitz matrices associated with a matrix-valued function. Moreover, we summarise the multigrid convergence results for such structures in Subsection \ref{ssec:proposal_block_circ}. Section \ref{sec:proofs} establishes a theoretical foundation for the proposed method. In particular, in Subsection \ref{ssec:gtoperator_and_symbol} we define the aggregated symbol-based grid transfer operator and analyse the properties of the matrices at coarser levels. In Subsection \ref{ssec:smoothing_prop} we introduce a block Jacobi smoother in the block circulant setting and provide conditions on the smoothing parameter such that the smoothing property is rigorously proven. Moreover, in Subsection \ref{ssec:approximation_prop} we establish the convergence of the two-grid method by demonstrating its approximation property. An extension of the convergence analysis to the V-cycle method, ensuring the effectiveness of our approach across a grid hierarchy, is presented in Subsection \ref{ssec:v-cycle}. Section \ref{sec:num_experiments} is dedicated to numerical experiments with problem descriptions in Subsection \ref{ssec:example_descr}, a comparison between the scalar and block Jacobi smoothers in Subsection \ref{ssec:scalar_vs_block}, a validation of the theory on the aggregated multigrid approach in Subsection \ref{ssec:aggregation_experiments}, an overview of the over-relaxation strategy in Subsection \ref{ssec:over_relaxation}, and a comparison of the performances of all the presented multigrid methods as preconditioners in Subsection \ref{ssec:PPCG}. Section \ref{sec:conclusions} contains final remarks and future lines of research.


\section{Multigrid methods for structured linear systems}\label{sec:two_grid}
In the first part of the section we summarize the main components and properties of classical multigrid methods when applied to solve linear systems of the form
\[A_{n}x_{n}=b_{n},\]
where $A_{n}\in\mathbb{C}^{n\times n}$ is a generic positive definite matrix {\cite{Trot}}.
Then, we report the most recent convergence results when the coefficient matrix $A_{n}$ in addition is a block Toeplitz matrix associated with a matrix-valued symbol.
In the whole paper we use the following norm notation.
Given $1\le p<\infty$ and a vector $x \in \mathbb{C}^{n}$, we denote by $\|x\|_{p}$ the $p$-norm of $x$ and by $\|\cdot\|_{p}$ the associated induced matrix norm over $\mathbb{C}^{n\times n}$. If $X$ is positive definite,
$\|{v}\|_{X}={\|}X^{1/2}{v}{\|}_{2}$ {(resp. $\|Y\|_{X}=\|X^{1/2}YX^{-1/2}\|_{2}$)} denotes the Euclidean norm
weighted by $X$ on $\mathbb{C}^{n}$ { (resp. on $\mathbb{C}^{n\times n}$)}. Moreover, given a matrix-valued function  $\mathbf{f} \in L^p({Q})$ (all its components $f_{ij}:Q\to\mathbb C,\
i,j=1,\ldots, d $ belong to $L^p({Q})$) we define
$ \|\mathbf{f} \|_{\infty}={\rm ess\,sup}_{\theta\in Q}\|\mathbf{f} (\theta)\|_2$.

When considering a multigrid method with only two grids, a TGM procedure is the combination of a stationary iterative method, the pre/post smoother, and a full rank rectangular matrix $P_{n,k}\in\mathbb{C}^{n\times k}$, $k<n$, the coarse grid operator. 
Precisely, if the smoothers $\mathcal{V}_{n,\rm{pre}}$ and $\mathcal{V}_{n,\rm{post}}$ have iteration matrices ${V}_{n,\rm{pre}}$ and ${V}_{n,\rm{post}}$, one iteration of the TGM is described by Algorithm 1.

\begin{small}
\begin{algorithm}
	\caption{TGM$(A_{n},\mathcal{V}_{n,\rm{pre}}^{\nu_{\rm{pre}}},\mathcal{V}_{n,\rm{post}}^{\nu_{\rm{post}}},P_{n,k},b_{n},x_{n}^{(j)})$}
	\label{alg:Two_grid}
	\begin{algorithmic}
		\STATE{ 0. $\tilde{x}_{n}=\mathcal{V}_{n,\rm{pre}}^{\nu_{\rm{pre}}}(A_{n},{b}_{n},x_{n}^{(j)})$}
		\STATE{ 1. $r_{n}=b_{n}-A_{n}\tilde{x}_{n}$}
		\STATE{ 2. $r_{k}=P_{n,k}^{H}r_{n}$}
		\STATE{ 3. $A_{k}=P_{n,k}^{H}A_{n}P_{n,k}$}
		\STATE{ 4. Solve $A_{k}y_{k}=r_{k}$}
		\STATE{ 5. $\hat{x}_{n}=\tilde{x}_{n}+P_{n,k}y_{k}$}
		\STATE{ 6. $x_{n}^{(j+1)}=\mathcal{V}_{n,\rm{post}}^{\nu_{\rm{post}}}(A_{n},{b}_{n},\hat{x}_{n})$}
		
	\end{algorithmic}
\end{algorithm}
\end{small}

The steps $1.\rightarrow5.$ define the ``coarse grid correction'' that depends on the projecting operator $P_{n,k}$, while step $0.$ and
step $6.$ consist, respectively, in applying $\nu_{\rm{pre}}$ times a pre-smoother
and $\nu_{\rm{post}}$ times a post-smoother of the given iterative methods.
Step 3. defines the coarser matrix $A_{k}$ according to the Galerkin approach. 

 A complete V-cycle procedure is obtained replacing the direct solution at step 4. with a recursive call of the TGM applied to the coarser linear system $A_{k_{\ell}}y_{k_{\ell}}=r_{k_{\ell}}$, where $\ell$ represents the level. The recursion stops at level ${\ell_{\min}}$ when $k_{\ell_{\min}}$ becomes small enough for solving cheaply step 4. with a direct solver.
 
The TGM algorithm can be seen as a stationary method itself with the following iteration matrix
\begin{small}
\begin{align*}
{\rm TGM}(A_{n},V_{n,\rm{pre}}^{\nu_{\rm{pre}}},V_{n,\rm{post}}^{\nu_{\rm{post}}},P_{n,k})=
V_{n,\rm{post}}^{\nu_{\rm{post}}}
\left[I_{n}-P_{n,k}\left(P_{n,k}^{H}
A_{n}P_{n,k}\right)^{-1}P_{n,k}^{H}A_{n}\right]V_{n,\rm{pre}}^{\nu_{\rm{pre}}}.
\end{align*}
\end{small}
Consequently, a pivotal convergence result can be expressed as in the following theorem.

\begin{theorem}\label{thm:Rstub}(\cite{RStub})
	Let $A_{n}$ be a positive definite matrix of size $n$ and let $V_{n,{\rm post}},$ $V_{n,{\rm pre}}$ be defined as in the {\rm TGM} algorithm.
	Assume
	\begin{itemize}
	\begin{small}
		\item[(a)] $\exists a_{\rm{pre}}>0\,:\;\|V_{n,\rm{pre}}x_{n}\|_{A_{n}}^{2}\leq\|x_{n}\|_{A_{n}}^{2}- a_{\rm{pre}}\|V_{n,\rm{pre}}x_{n}\|_{A_{n}^2}^{2},\qquad \forall x_{n}\in\mathbb{C}^{n},$
		\item[(b)] $\exists a_{\rm{post}}>0\,:\;\|V_{n,\rm{post}}x_{n}\|_{A_{n}}^{2}\leq\|x_{n}\|_{A_{n}}^{2}- a_{\rm{post}}\|x_{n}\|_{A_{n}^2}^{2},\qquad \forall x_{n}\in\mathbb{C}^{n},$
		\item[(c)] $\exists\gamma>0\,:\;\min_{y\in\mathbb{C}^{k}}\|x_{n}-P_{n,k}y\|_{2}^{2}\leq \gamma\|x_{n}\|_{A_{n}}^{2},\qquad \forall x_{n}\in\mathbb{C}^{n}.$
	\end{small}
	\end{itemize}
	Then $\gamma\geq a_{\rm{post}}$ and
\begin{small}
	\begin{align*}
	\|{\rm TGM}(A_{n},V_{n,\rm{pre}},V_{n,\rm{post}},P_{n,k})\|_{A_{n}}\leq\sqrt{\frac{1- a_{\rm{post}}/ \gamma}{1+ a_{\rm{pre}}/ \gamma}}<1.
	\end{align*}
\end{small}
\end{theorem}

We highlight that such result splits the assumptions that should be fulfilled by the smothers: conditions {$(a)-(b)$} called ``smoothing properties'' and the assumption involving only $P_{n,k}$: condition  $(c)$  called ``approximation property''.

Moreover $ a_{\rm{post}}$ and $\gamma$ are independent of $n$, then a TGM verifying Theorem \ref{thm:Rstub}  exhibits a linear convergence. That is, the number of iterations in order to reach a given accuracy $\epsilon$ can be bounded from above by a constant independent of $n$ (possibly depending on the parameter $\epsilon$).

 \subsection{Multigrid methods for block-circulant and block-Toeplitz matrices}\label{ssec:proposal_block_circ}
In the present subsection we recall how the previous results can be written for block-circulant or block-Toeplitz matrix associated with a matrix-valued function. Then, we briefly describe the structure of such matrices,  restricting to the context of interest for the present paper. Both block-Toeplitz and block-circulant are associated with a function $\mathbf{f}:Q\to \mathbb{C}^{d\times d}$,  $Q=(-\pi,\pi)$ such that $\mathbf{f} \in L^p([-\pi,\pi])$ and the Fourier coefficients are
	\begin{align*}
	\hat{\mathbf{f}_j}:=\frac1{2\pi}\int_{Q}\mathbf{f}(\theta){\rm e}^{- \iota j \theta} d\theta\in\mathbb{C}^{d\times d}\qquad  \iota^2=-1, \, j\in\mathbb Z.
	\end{align*}
	The block-Toeplitz matrix associated with {$\textbf{f}$ is the matrix with $d$ blocks of size $n$ and hence it has order $d\cdot n$} given by
\begin{small}
	\begin{align*}
	T_n(\mathbf{f})=\sum_{|j|<n}J_{n}^{(j)}\otimes\hat{\mathbf{f}_j},
	\end{align*}
\end{small}where $\otimes$ denotes the (Kronecker) tensor product of matrices. The term $J_n^{(j)}$ is the matrix of order $n$ whose $(i,k)$ entry equals $1$ if $i-k=j$ and zero otherwise.  

The set $\{T_n(\mathbf{f})\}_{n\in\mathbb N}$ is
called the \textit{family of block-Toeplitz matrices generated by $\mathbf{f}$}, that
in turn is referred to as the \textit{generating function or the symbol of
	$\{T_n(\mathbf{f})\}_{n\in\mathbb N}$}.
If $\mathbf{f}$ is a  matrix-valued trigonometric polynomial, then the block-circulant matrix of order $dn$ generated by $\mathbf{f}$ can be decomposed as
\begin{small}
  \begin{equation*}
\mathcal{C}_{n}(\mathbf{f})=(F_n\otimes I_{d}) \diagi(\mathbf{f}(\theta_i^{(n)}))(F_{n}^{H}\otimes I_{d}), \quad   \theta_i^{(n)}=\frac{2\pi i}{n}, \, i \in \mathcal{I}_{n} = \{0,\ldots,n-~1\},
\end{equation*} 
  \end{small}  
where $\diagi(\mathbf{f}(\theta_i^{(n)}))$ is the block-diagonal matrix where the block-diagonal elements are  $\mathbf{f}(\theta_i^{(n)})$.

For the convergence analysis of block-circulant and block-Toeplitz matrices, convergence results are based on the Ruge-St\"uben Theorem \ref{thm:Rstub}, see \cite{CCS, HS2, MR4284081}. 
The smoothing properties are satisfied with specific choices of the smoother parameter of damped Richardson or Jacobi methods see \cite[Lemma 1]{MR4284081}.
The approximation requires a precise definition of $P_{n,k}$ and it slightly changes depending whether we are in the Toeplitz or Circulant setting.

That is,
\begin{small}
\begin{equation}\label{eqn:def_Pnkd}
	P^d_{n,k}=\mathcal{C}_{n}(\mathbf{p})(K^{Odd}_{n,k} \otimes I_d).\quad P^d_{n,k}=\mathcal{T}_{n}(\mathbf{p})(K^{Even}_{n,k} \otimes I_d).
\end{equation}
\end{small}
where $K^{Odd}_{n,k}$ is a $n\times k$ matrix obtained by removing the even rows from the identity matrix of size $n$,  keeping the odd rows. On the other hand, $K^{Even}_{n,k}$ keeps the even rows.  The key point in such grid transfer operators is that preserves the block structure at the coarser levels. 
The convergence results in the block structured setting were derived in \cite{MR4389580} exploiting the block-symbol analysis.

In detail the setting is the following. We suppose that there exist unique $\theta_0\in[0,2\pi)$ and $\bar{\jmath} \in\{1,\dots,d\}$ such that 
\begin{small}
\begin{equation}\label{eqn:condition_on_f}
\left\{\begin{array}{ll}
\lambda_j(\mathbf{f}(\theta))=0, & \mbox{for } \theta=\theta_0 \mbox{ and } j=\bar{\jmath}, \\
\lambda_j(\mathbf{f}(\theta))>0, & {\rm otherwise}.
\end{array}\right.
\end{equation}
\end{small}
That is the matrix-valued function $\mathbf{f}(\theta)$ has exactly one zero eigenvalue in $\theta_0$ and it is positive definite in $[0,2\pi)\backslash\{\theta_0\}$. As a consequence, the associated  block circulant matrix could be singular and the ill-conditioned subspace is the eigenspace associated with $\lambda_{\bar{\jmath}}(\mathbf{f}(\theta_0))$. Moreover, the block-Toeplitz matrices $T_{n}(\mathbf{f})$ are positive definite with the same ill-conditioned subspace and become ill-conditioned as $N$ increases. The key point in \cite{MR4389580} is that $\mathbf{f}(\theta)$ is can be diagonalized by an orthogonal matrix $Q(\theta)$
\begin{small}
\begin{equation}\label{eq:dec_f}
\begin{split}
& \mathbf{f}(\theta) = Q(\theta)D(\theta)Q(\theta)^H =\\
& \left[
\begin{array}{@{\;}c@{\;}|@{\;}c@{\;}|@{\;}c@{\;}|@{\;}c@{\;}|@{\;}c@{\;}|@{\;}c@{\;}}
q_1(\theta)&\dots &q_{\bar{\jmath}}(\theta)&\dots& q_d(\theta)
\end{array}
\right]\begin{bmatrix}
\lambda_1(\mathbf{f}(\theta))& & &  & &\\
& \ddots & &  &\\
& & \lambda_{\bar{\jmath}}(\mathbf{f}(\theta))& & \\
& & & \ddots & \\
& & & &\lambda_d(\mathbf{f}(\theta))
\end{bmatrix}
\left[
\begin{array}{ccccccc}
{q_1}^H(\theta) \\
\hline
\vdots\\
\hline
q_{\bar{\jmath}}^H(\theta)\\
\hline
\vdots\\
\hline
{q_d}^H(\theta)
\end{array}
\right],
\end{split}
\end{equation}
\end{small}
where $q_{\bar{\jmath}}(\theta)$ is the normalized eigenvector that generates the ill-conditioned subspace since $q_{\bar{\jmath}}(\theta_0)$ is the eigenvector of $\mathbf{f}(\theta_0)$ associated with $\lambda_{\bar{\jmath}}(\mathbf{f}(\theta_0))=0$. 
Under the following  assumptions, the  sufficient conditions to ensure the linear convergence of the TGM are choosing   $\mathbf{p}$ such that
\begin{enumerate}
\item [$(i)$]
\begin{small}
\begin{equation*}  
	\mathbf{p}(\theta)^{H}\mathbf{p}(\theta)+\mathbf{p}(\theta+\pi)^{H}\mathbf{p}(\theta+\pi)>0 \quad\forall \theta\in[0,2\pi), 
\end{equation*}
\end{small}
which implies that the trigonometric function
\begin{equation}\label{eqn:def_s}
\textbf{s}(\theta) = \mathbf{p}(\theta)\left(\mathbf{p}(\theta)^{H}\mathbf{p}(\theta)+\mathbf{p}(\theta+\pi)^{H}\mathbf{p}(\theta+\pi)\right)^{-1}\mathbf{p}(\theta)^{H}
\end{equation}
is well-defined for all $\theta\in[0,2\pi)$,
\item [$(ii)$] \begin{small}
\[	\textbf{s}(\theta_0)q_{\bar{\jmath}}(\theta_0) = q_{\bar{\jmath}}(\theta_0),\]
\end{small}
\item [$(iii)$] \begin{small}
 \[
\lim_{\theta\rightarrow \theta_0} \lambda_{\bar{\jmath}}(\mathbf{f}(\theta))^{-1}(1-\lambda_{\bar{\jmath}}(\textbf{s}(\theta)))=c, \quad c\in\mathbb{R}.
\]
\end{small}
\end{enumerate}
 Conditions $(i)--(iii)$ can be further simplified and some results suggest how to deal with their validation \cite[Lemma 4.3, Lemma 4.4, Lemma 4.6]{MR4389580}. Moreover,   under some additional hypotheses on $\mathbf{p}$ and $\mathbf{f}$ convergence and optimality  when dealing with V-cycle with more than two grids can be derived \cite[Lemma 4.7-4.8]{MR4389580}.
Note that  condition $(i)$ does not depend on $\textbf{f}$ and its spectral properties, $(ii)$ and $(iii)$ depend on the eigenvector associated to the singularity of $\textbf{f}$ and the behaviour of the $\bar{j}$-the eigenvalue function of  $\textbf{f}$.  
Even though the singularity of $\textbf{f}$ is known, in some cases describing the behaviour of the minimal eigenvalue function of  $\textbf{f}$ is not immediate. Theorem \ref{thm:behav_coarse_symbol} provides a spectral result that can be exploited also to simplify the analysis of $\lambda_{\bar{j}}(\mathbf{f}(\theta))$. 

\section{Aggregation and Block Smoothers: MGM Convergence analysis}\label{sec:proofs}
In this Section we show how the symbol-based convergence analysis can be performed and produces different advantages in the context of multigrid procedures based on the aggregation of the unknowns combined with proper block smoothers.  In particular, we focus on  block Jacobi smoothers and we show how the choice of the relaxation parameter can be related to the properties of the matrix-valued symbols. Additional simplifications can be derived when we can exploit the structure of the Fourier coefficients of the trigonometric polynomial $\mathbf{f}$.  With a full knowledge of the properties of the symbol at the coarser levels we are able to prove the convergence and optimality of the TGMs at each level, which leads to the convergence of the V-cycle method. Precisely in Section \ref{ssec:gtoperator_and_symbol} we introduce the grid transfer operator which performs the reduction of the problem from block to scalar and we show the form and the properties of the associated scalar symbol. We present in Theorem \ref{thm:behav_coarse_symbol} a result that relates the behaviour of the eigenvalue function associated with the singularity of $\mathbf{f}$ with the symbol at the coarser level, which is helpful both for the development of a multilevel strategy in the aggregation context and for the theoretical analysis for the block preserving methods presented in \cite{MR4389580,MR4284081}. The range of admissible values that leads to the validation of the smoothing property is presented in Subsection \ref{ssec:smoothing_prop} with the related simplifications presented in Theorem \ref{th:simplifyrank2smoother}.  The validation of the approximation property (c) for the aggregation-based grid transfer operator is shown in Subsection \ref{ssec:approximation_prop}. Moreover, we present the V-cycle strategy and the convergence analysis based on the symbols at the coarser levels in Subsection \ref{ssec:v-cycle}. Finally, we conclude the Section with a discussion on the choice of optimal parameters in Subsection \ref{ssec:optimal_parameters}.
\subsection{Grid Transfer Operator and Symbol at the Coarser Level}\label{ssec:gtoperator_and_symbol}
In this subsection we consider a grid transfer operator for a two-grid method for a linear system with matrix $\mathcal{C}_{n}({\mathbf{f}})$.
Recalling the decomposition of the matrix-valued trigonometric polynomial $\mathbf{f}$ in equation \eqref{eq:dec_f},  we define
\begin{equation}
\label{eq:aggr_P_I}
   P^d_{n,k}=I_{n}\otimes q_{\bar{\jmath}}(\theta_0).
\end{equation}
Using the two-grid approach described in Section \ref{sec:two_grid}, the operator at the coarser level becomes a circulant matrix generated by a scalar valued trigonometric polynomial. A detailed clarification can be found in the succeeding lemma.
\begin{lemma}\label{lem:coarse_symbol}
    Let $\mathbf{f}$ be defined as in Section \ref{sec:two_grid} and associated to the matrix $\mathcal{C}_{n}({\mathbf{f}})$. Let $ P^d_{n,k}$ be the grid transfer operator defined in (\ref{eq:aggr_P_I}). Then,
  \begin{equation}\label{eq:PHAP}
\left(P^d_{n,k}\right)^H\mathcal{C}_{n}({\mathbf{f}})P^d_{n,k}= \mathcal{C}_{n}(\tilde{f}) 
\end{equation} 
with 
\begin{equation}\label{eq:f_tilde}
	\tilde{f}(\theta)=q_{\bar{\jmath}}^H(\theta_0) \mathbf{f}(\theta) q_{\bar{\jmath}}(\theta_0).
\end{equation}
\end{lemma}
\begin{proof}
Decomposing $\mathcal{C}_{n}(\mathbf{f})$ as
\[\mathcal{C}_{n}(\mathbf{f})=(F_n\otimes I_{d}) \diagi(\mathbf{f}(\theta_i^{(n)}))(F_{n}^{H}\otimes I_{d}),\]
we can write
\begin{align*}
   \left(P^d_{n,k}\right)^H\mathcal{C}_{n}({\mathbf{f}})P^d_{n,k}&= \left(I_{n}\otimes q_{\bar{\jmath}}^H(\theta_0)\right)
        (F_n\otimes I_{d}) \diagi\left(\mathbf{f}(\theta_i^{(n)})\right)(F_{n}^{H}\otimes I_{d})
        \left(I_{n}\otimes q_{\bar{\jmath}}(\theta_0)\right)\\
        &= \left(F_{n}\otimes q_{\bar{\jmath}}^H(\theta_0)\right)\diagi\left(\mathbf{f}(\theta_i^{(n)})\right)\left(F_{n}^{H}\otimes q_{\bar{\jmath}}(\theta_0)\right)\\
        &= (F_n\otimes I_{d})\left( I_n \otimes q_{\bar{\jmath}}^H(\theta_0)\right) \diagi(\mathbf{f}(\theta_i^{(n)}))\left( I_n \otimes q_{\bar{\jmath}}(\theta_0)\right)(F_{n}^{H}\otimes I_{d})
\end{align*}
The thesis follows from the equalities
\begin{align*}
    \left( I_n \otimes q_{\bar{\jmath}}^H(\theta_0)\right) &\diagi\left(\mathbf{f}(\theta_i^{(n)})\right)\left( I_n \otimes q_{\bar{\jmath}}(\theta_0)\right) =  \\
    &= \left( I_n \otimes q_{\bar{\jmath}}^H(\theta_0)\right) 
    \begin{bmatrix}
        \mathbf{f}(\theta_0^{(n)}) & & \\
        &\ddots&\\
        && \mathbf{f}(\theta_{n-1}^{(n)})
    \end{bmatrix}
    \left( I_n \otimes q_{\bar{\jmath}}(\theta_0)\right)\\
    &= 
    \begin{bmatrix}
        q_{\bar{\jmath}}^H(\theta_0)\mathbf{f}(\theta_0^{(n)})q_{\bar{\jmath}}(\theta_0) & & \\
        &\ddots&\\
        && q_{\bar{\jmath}}^H(\theta_0)\mathbf{f}(\theta_{n-1}^{(n)})q_{\bar{\jmath}}(\theta_0)
    \end{bmatrix}
    \\
    & = \diagi\left(\tilde{f}(\theta_i^{(n)})\right).\\
\end{align*}
\end{proof} 

The following theorem, though not arduous to establish, plays a pivotal role in the context of this paper, providing fundamental insights for the subsequent analysis.

\begin{theorem}\label{thm:behav_coarse_symbol}
    Let $\mathbf{f}$ be defined as in Section \ref{sec:two_grid} and define $\tilde{f}$ as in equation \eqref{eq:f_tilde}. Suppose $\lambda_{\bar{\jmath}}(\mathbf{f})$ vanishes in $\theta_0$ with a zero of order $\beta$, then $\tilde{f}$ vanishes in $\theta_0$ (and only in $\theta_0$) with a zero of order $\beta$.
\end{theorem}
\begin{proof} 
    The function $\tilde{f}$ vanishes in $\theta_0$ because $\tilde{f}(\theta_0)=\lambda_{\bar{\jmath}}(\mathbf{f}(\theta_0))q_{\bar{\jmath}}^H(\theta_0) q_{\bar{\jmath}}(\theta_0)=0$. The order of the zero can be determined by computing the following limit:
    \begin{align*}
        \lim_{\theta \rightarrow \theta_0} \frac{\tilde{f}(\theta)}{\lambda_{\bar{\jmath}}(\mathbf{f}(\theta))}
        &= \lim_{\theta \rightarrow \theta_0} \frac{q_{\bar{\jmath}}^H(\theta_0) \mathbf{f}(\theta) q_{\bar{\jmath}}(\theta_0)}{\lambda_{\bar{\jmath}}(\mathbf{f}(\theta))}
        = \lim_{\theta \rightarrow \theta_0} \frac{q_{\bar{\jmath}}^H(\theta) \mathbf{f}(\theta) q_{\bar{\jmath}}(\theta)}{\lambda_{\bar{\jmath}}(\mathbf{f}(\theta))}
        = \lim_{\theta \rightarrow \theta_0} \frac{\lambda_{\bar{\jmath}}(\mathbf{f}(\theta))q_{\bar{\jmath}}^H(\theta) q_{\bar{\jmath}}(\theta)}{\lambda_{\bar{\jmath}}(\mathbf{f}(\theta))}\\
        &=q_{\bar{\jmath}}^H(\theta_0)q_{\bar{\jmath}}(\theta_0)\neq 0.
    \end{align*}
    Subsequently, if $\lambda_{\bar{\jmath}}(\mathbf{f})$ vanishes in $\theta_0$ with a zero of order $\beta$, then $\tilde{f}$ vanishes in $\theta_0$ with a zero of order~$\beta$.

    Finally, we prove $\tilde{f}(\theta)\neq 0$ if $\theta\neq \theta_0$ by contradiction. Suppose $\tilde{f}(\tilde{\theta})= 0$ for a $\tilde{\theta}\neq \theta_0$. Then, $q_{\bar{\jmath}}^H(\theta_0) \mathbf{f}\left(\tilde{\theta}\right) q_{\bar{\jmath}}(\theta_0)=0$, but this is absurd because $\mathbf{f}\left(\tilde{\theta}\right)$ is HPD for \eqref{eqn:condition_on_f}.
\end{proof} 

\begin{remark}

Theorem \ref{thm:behav_coarse_symbol} is a result that can be exploited to simplify the analysis of the behaviour of $\lambda_{\bar{j}}(\mathbf{f}(\theta))$. Indeed, when dealing with condition $(iii)$ and its simplified versions, we can avoid to spectrally study the matrix-valued symbol. We can focus directly on the scalar-valued function $ \tilde{f}(\theta)$   which has the same spectral behaviour of  $\lambda_{\bar{j}}(\mathbf{f}(\theta))$ in $\theta_{0}$.
\end{remark}

\subsection{Smoothing Property for Block Jacobi}\label{ssec:smoothing_prop}
We consider the block Jacobi method as smoother and we choose the relaxation parameter $\omega$, depending on the properties of $\mathbf{f}$, such that smoothing property (a) of Theorem \ref{thm:Rstub} is satisfied.
Considering the matrix $A_{N}:=\mathcal{C}_{n}({\mathbf{f}})$, the iteration matrix of the relaxed block Jacobi method has the form
\begin{equation}
\label{eq:defVn_block_Jacobi}
    V_{n,\rm{post}} = I_{N}-\omega D_B^{-1}A_{N},
\end{equation}
where $D_B$ is a block diagonal matrix with the same block diagonal as $A_{N}$, that is
\begin{equation}
\label{eq:D_B_structure}
  D_B = I_n \otimes \hat{\mathbf{f}}_0 = \mathcal{C}_{n}(\hat{\mathbf{f}}_0),
\end{equation}
where $\hat{\mathbf{f}}_0$ is the 0th Fourier coefficient of $\mathbf{f}$.
\begin{theorem} \label{thm:smoothing_par}
Consider the matrix $A_{N}:=\mathcal{C}_{n}({\mathbf{f}})$, with $\mathbf{f}$ $d\times d$ matrix-valued trigonometric polynomial, $\mathbf{f}\geq0$. 
	Let $  V_{n,\rm{post}}$ defined in (\ref{eq:defVn_block_Jacobi}) be the iteration matrix of the relaxed block Jacobi method applied to $A_{N}$. If 
\begin{equation}
\label{eq:condition_smoother}
   0<\omega< \frac{2}{\left\|\hat{\mathbf{f}}_0^{-\frac{1}{2}} \mathbf{f}\hat{\mathbf{f}}_0^{-\frac{1}{2}}\right\|_{\infty}},
\end{equation}
then there exists a positive value $a_{\rm post}$ independent of $n$ such that inequality $(a)$ in Theorem \ref{thm:Rstub} is satisfied.
\end{theorem}

\begin{proof}
By definition (\ref{eq:defVn_block_Jacobi}), relation (a) of Theorem \ref{thm:Rstub} corresponds to prove that exists $a_{\rm post}>0$ such that the matrix

\begin{equation*}
A_{N}- a_{\rm post} A_{N}^2 - \left( I_{N}-\omega D_B^{-1}A_{N}\right)^HA_{N}\left( I_{N}-\omega D_B^{-1}A_{N}\right)      
\end{equation*}
is SPD. Equivalently, by computation, 

\begin{equation*}
2\omega D_B^{-1}-\omega^2 D_B^{-1} A_{N}D_B^{-1}-  a_{\rm post}I_N \ge 0.     
\end{equation*}
We can write the latter condition in terms of the eigenvalues $\lambda$ of $A_{N}$
\begin{equation*}
 a_{\rm post}\le \lambda\left(2\omega D_B^{-1}-\omega^2 D_B^{-1} A_{N}D_B^{-1}\right).     
\end{equation*}
Since $ a_{\rm post}$ should be strictly positive, we need to choose $\omega$ such that 
\begin{equation*}
    \lambda \left(2\omega D_B^{-1}-\omega^2 D_B^{-1} A_{N}D_B^{-1}\right)>0.
\end{equation*}
By the Sylvester inertia law, we look for $\omega$ such that
\begin{equation*}
    \lambda \left(2\omega I_N-\omega^2 D_B^{-\frac{1}{2}} A_{N} D_B^{-\frac{1}{2}} \right)>0.
\end{equation*}
The latter is implied if we take 
\begin{equation}
\label{eq:choiche_omega_matrix}
    0<\omega< \frac{2}{\lambda_{\max}\left(D_B^{-\frac{1}{2}} A_{N}D_B^{-\frac{1}{2}} \right)}=
    \frac{2}{\lambda_{\max}\left(\mathcal{C}_n\left(\hat{\mathbf{f}}_0^{-\frac{1}{2}} \mathbf{f}\hat{\mathbf{f}}_0^{-\frac{1}{2}}\right) \right)}.
\end{equation}
To conclude the proof we rewrite (\ref{eq:choiche_omega_matrix}) in terms of the 
generating functions:
\begin{equation*}
    0<\omega< \frac{2}{\left\|\hat{\mathbf{f}}_0^{-\frac{1}{2}} \mathbf{f}\hat{\mathbf{f}}_0^{-\frac{1}{2}}\right\|_{\infty}}.
\end{equation*}

\end{proof}

 The latter theorem implies that we need to study the structure and the eigenvalues of  $\hat{\mathbf{f}}_0^{-\frac{1}{2}} \mathbf{f}(\theta)\hat{\mathbf{f}}_0^{-\frac{1}{2}}$, depending on $\theta$, and then compute its maximum with respect to $\theta$.

However, condition (\ref{eq:condition_smoother}) simplifies when the generating function $\mathbf{f}$ of $\mathcal{C}_{n}(\mathbf{f})$ possesses a particular structure.

\begin{theorem}
\label{th:simplifyrank2smoother}
Consider the matrix $A_{N}:=\mathcal{C}_{n}({\mathbf{f}})$, where $\mathbf{f}$ $d\times d$ matrix-valued trigonometric polynomial of degree 1, with $\mathbf{f}\geq0$. Assume $\mathbf{f}(\theta)= \hat{\mathbf{f}}_0+\hat{\mathbf{f}}_{-1}{\rm e}^{-\mathbf{i}\theta}+\hat{\mathbf{f}}_{1}{\rm e}^{\mathbf{i}\theta}$,
with
\[\hat{\mathbf{f}}_{-1}=uv^T; \quad \hat{\mathbf{f}}_{1}=(\hat{\mathbf{f}}_{-1})^T=vu^T, \qquad u,v\in \mathbb{C}^d.\]
Then,
    \[\left\|\hat{\mathbf{f}}_0^{-\frac{1}{2}} \mathbf{f}\hat{\mathbf{f}}_0^{-\frac{1}{2}}\right\|_{\infty}=
    1+ \max_\theta \left(z^Tw\cos\theta+ \sqrt{(z^Tw)^2(\cos^2\theta-1)+\|w\|^2\|z\|^2}\right)\]
with $w=\hat{\mathbf{f}}_0^{-\frac{1}{2}} u$ and $z=\hat{\mathbf{f}}_0^{-\frac{1}{2}} v$.

\end{theorem}
\begin{proof}
Exploiting the structure of $\hat{\mathbf{f}}_{-1}$ and the symmetry of $ \hat{\mathbf{f}}_0^{-\frac{1}{2}}$ , we obtain
\begin{equation*}
\begin{split}
  \hat{\mathbf{f}}_0^{-\frac{1}{2}} \mathbf{f}(\theta)\hat{\mathbf{f}}_0^{-\frac{1}{2}} &= \hat{\mathbf{f}}_0^{-\frac{1}{2}} \left( \hat{\mathbf{f}}_0+\hat{\mathbf{f}}_{-1}{\rm e}^{-\mathbf{i}\theta}+\hat{\mathbf{f}}_{1}{\rm e}^{\mathbf{i}\theta}\right)\hat{\mathbf{f}}_0^{-\frac{1}{2}}\\
  &= I_d +\left(\hat{\mathbf{f}}_0^{-\frac{1}{2}} u\right)\left(v^T\hat{\mathbf{f}}_0^{-\frac{1}{2}} \right){\rm e}^{-\mathbf{i}\theta}
  +\left(\hat{\mathbf{f}}_0^{-\frac{1}{2}} v\right)\left(u^T\hat{\mathbf{f}}_0^{-\frac{1}{2}} \right){\rm e}^{\mathbf{i}\theta}\\
  &=  I_d+
  \left(\hat{\mathbf{f}}_0^{-\frac{1}{2}} u\right)\left(\hat{\mathbf{f}}_0^{-\frac{1}{2}}v \right)^T{\rm e}^{-\mathbf{i}\theta}
    +\left(\hat{\mathbf{f}}_0^{-\frac{1}{2}} v\right)\left(\hat{\mathbf{f}}_0^{-\frac{1}{2}} u \right)^T{\rm e}^{\mathbf{i}\theta}.
  \end{split}
\end{equation*}
Setting $w=\hat{\mathbf{f}}_0^{-\frac{1}{2}} u$ and $z=\hat{\mathbf{f}}_0^{-\frac{1}{2}} v$, the function $ \hat{\mathbf{f}}_0^{-\frac{1}{2}} \mathbf{f}\hat{\mathbf{f}}_0^{-\frac{1}{2}}$ can be rewritten as

\begin{equation*}
 \hat{\mathbf{f}}_0^{-\frac{1}{2}} \mathbf{f}(\theta)\hat{\mathbf{f}}_0^{-\frac{1}{2}}= I_d+ A(\theta),
\end{equation*}
with $A(\theta)=wz^T{\rm e}^{-\mathbf{i}\theta}
    +zw^T{\rm e}^{\mathbf{i}\theta}$.

Consequently, the eigenvalues of $  \hat{\mathbf{f}}_0^{-\frac{1}{2}} \mathbf{f}(\theta)\hat{\mathbf{f}}_0^{-\frac{1}{2}}$ are of the form
\begin{equation*}
\left\{1+\lambda_j\left(A(\theta)\right)\right\}.
\end{equation*}

For all $\theta$, $A(\theta)$ is a matrix with rank at most equal to 2 and we denote by 
  $\lambda_1(A(\theta))$ and $\lambda_2(A(\theta))$ the (possibly) non-zero eigenvalues. The values $\lambda_1(A(\theta))$ and $\lambda_2(A(\theta))$ can be computed exploiting the relations $\lambda_1(A)\lambda_2(A(\theta))=\frac{Tr\left(A(\theta)\right)^2-Tr\left(A(\theta)^2\right)}{2}$ and $\lambda_1(A(\theta))+\lambda_2(A(\theta))=Tr\left(A(\theta)\right)$ which are

\begin{equation*}
        \lambda_1(A(\theta))\lambda_2(A(\theta))=  (z^Tw)^2-||w||^2||z||^2; \quad \lambda_1(A(\theta))+\lambda_2(A(\theta))=2z^Tw\cos\theta,
    \end{equation*}
    obtaining 
    \begin{equation*}
        \lambda_{1,2}(A(\theta))= z^Tw\cos\theta\pm \sqrt{(z^Tw)^2(\cos^2\theta-1)+||w||^2||z||^2}
    \end{equation*}
    and 
    \begin{equation*}
        \max_\theta \lambda_{\max}\left(\hat{\mathbf{f}}_0^{-\frac{1}{2}} \mathbf{f}(\theta)\hat{\mathbf{f}}_0^{-\frac{1}{2}}\right)= 1+ \max_\theta \left(z^Tw\cos\theta+\sqrt{(z^Tw)^2(\cos^2\theta-1)+||w||^2||z||^2}
 \right).
    \end{equation*}
The proof is complete applying the definition of infinity norm for matrix-valued function and exploit the fact that $\hat{\mathbf{f}}_0^{-\frac{1}{2}} \mathbf{f}\hat{\mathbf{f}}_0^{-\frac{1}{2}}$ is HPD. 
\end{proof}
\subsection{Approximation Property}\label{ssec:approximation_prop}

The current subsection contains the proof of the approximation property $(c)$ of Theorem \ref{thm:Rstub}  choosing as grid transfer operator the aggregation  matrix  $P_{n,k}^d$  defined in equation (\ref{eq:aggr_P_I}). Consequently, the combination of Theorem \ref{th:tgmopt} and the findings on the smoothing property of Subsection  \ref{ssec:smoothing_prop} imply the convergence and optimality of the TGM.

\begin{theorem} \label{th:tgmopt}
%
%

	Consider the matrix $A_{N}:=\mathcal{C}_{n}({\mathbf{f}})$, with $\mathbf{f}$ $d\times d$ matrix-valued trigonometric polynomial, $\mathbf{f}\geq0$, such that condition (\ref{eqn:condition_on_f}) is satisfied. 
	Let $P^d_{n,k}$ 
	be the projecting operator defined as in equation (\ref{eq:aggr_P_I}).
	Then, there exists a positive value $\gamma$ independent of $n$ such that inequality $(c)$ in Theorem \ref{thm:Rstub} is satisfied.
\begin{proof}
The first part of the proof takes inspiration from \cite[Theorem 5.2]{MR4284081}. We report all the details for completeness, uniforming the notation. We remind that in order to prove that there exists $\gamma>0$ independent of $n$ such that for any $x_{N}\in\mathbb{C}^{N}$
\begin{align}\label{condW}
	\min_{y\in\mathbb{C}^{K}}\|x_{N}-P^d_{n,k}y\|_{2}^{2}\leq \gamma\|x_{N}\|_{A_{N}}^{2},
\end{align}
we can choose a special instance of $y$ in such a way that the previous inequality is
	reduced to a matrix inequality in the sense of the partial ordering of the real space of
	 Hermitian matrices. 
  For any $x_N\in\mathbb{C}^N$, let $\overline{y}\equiv\overline{y}(x_{N})\in\mathbb{C}^{K}$ be defined as
	$
	\overline{y}=(P^d_{n,k})^{H}x_{N}.
$
	Therefore, (\ref{condW}) is implied by
	\begin{align*}
	\|x_{N}-P^d_{n,k}\overline{y}\|_{2}^{2}\leq \gamma\|x_{N}\|_{A_{N}}^{2},
	\end{align*}
	where the latter is equivalent to the matrix inequality $
	G_{N}(\mathbf{p})^{H}G_{N}(\mathbf{p})\leq \gamma A_{N}$	with $G_{N}(\mathbf{p})=I_{N}-P^d_{n,k}(P^d_{n,k})^{H}$. By construction, the matrix $G_N(\mathbf{p})$ is a Hermitian unitary projector, in fact $G_N(\mathbf{p})^{H}G_N(\mathbf{p})= G_{N}(\mathbf{p})^{2} = G_N(\mathbf{p})$. As a consequence, the preceding matrix inequality can be rewritten as
\begin{small}
	\begin{align}\label{rela}
	G_{N}(\mathbf{p})\leq \gamma \mathcal{C}_{n}({\mathbf{f}}).
	\end{align}
\end{small}
Then, we have to prove that there exists $\gamma>0$ such that
\begin{equation}
    \label{eq:dis_approx_prop}
    I_{d n} -P_{n,k}^d(P_{n,k}^d)^H\le \gamma \mathcal{C}_{n}({\mathbf{f}})
\end{equation}
i.e.
\begin{equation*}
    I_{d n} -I_{n}\otimes q_{\bar{\jmath}}(\theta_0)q_{\bar{\jmath}}(\theta_0)^H\le \gamma \mathcal{C}_{n}({\mathbf{f}})
\end{equation*}
i.e.
\begin{equation*}
    I_{d n} -I_{n}\otimes q_{\bar{\jmath}}(\theta_0)q_{\bar{\jmath}}(\theta_0)^H\le \gamma (F_n\otimes I_d)D_{n}(\mathbf{f}) (F_n^H\otimes I_d),
\end{equation*}
i.e.
\begin{equation*}
    I_{d n} -(F_n^H\otimes I_d)(I_{n}\otimes q_{\bar{\jmath}}(\theta_0)q_{\bar{\jmath}}(\theta_0)^H)(F_n\otimes I_d)\le \gamma D_{n}(\mathbf{f}) ,
\end{equation*}
i.e.
\begin{equation*}
    I_{d n} -I_{n}\otimes q_{\bar{\jmath}}(\theta_0)q_{\bar{\jmath}}(\theta_0)^H\le \gamma D_{n}(\mathbf{f}) ,
\end{equation*}
where $D_n(\mathbf{f})=diag_{j=0,\dots, n-1}\mathbf{f}\left(\theta_j^{(n)}\right)$ and $I_{n}\otimes q_{\bar{\jmath}}(\theta_0)q_{\bar{\jmath}}(\theta_0)^H=diag_{j=0,\dots, n-1}(q_{\bar{\jmath}}(\theta_0)q_{\bar{\jmath}}(\theta_0)^H)$. Consequently (\ref{eq:dis_approx_prop}) is equivalent to 
\begin{equation*}
    I_d-q_{\bar{\jmath}}(\theta_0)q_{\bar{\jmath}}(\theta_0)^H\le \gamma \mathbf{f}\left(\theta_j^{(n)}\right) \qquad \forall j=0,\dots, n-1.
\end{equation*}

The latter inequality is equivalent to prove that $\exists \gamma>0$ such that $\forall j=0,\dots, n-1$
\begin{equation*}
   w^H \left(\gamma \mathbf{f}\left(\theta_j^{(n)}\right)-I_d+ q_{\bar{\jmath}}(\theta_0)q_{\bar{\jmath}}(\theta_0)^H\right)w\ge 0 \quad \forall w \in \mathbb{C}^{d}, \quad w^Hw=1
\end{equation*}
which is equivalent to
\begin{equation}
\label{eq:dis_approx_prop_symb}
    \gamma w^H\mathbf{f}\left(\theta_j^{(n)}\right)w-1+ |q_{\bar{\jmath}}(\theta_0)^Hw|^2\ge 0 \quad \forall w \in \mathbb{C}^{d}, \quad w^Hw=1.
\end{equation}

Now we consider two cases: $\theta_j^{(n)}\neq\theta_0$ and $\theta_j^{(n)} = \theta_0$. In the first case with $\theta_j^{(n)}\neq\theta_0$, we have $\mathbf{f}\left(\theta_j^{(n)}\right)>0$.
Since $|q_{\bar{\jmath}}(\theta_0)^Hw|^2>0$, the latter is implied if we prove that  $\exists \gamma>0$ such that $\forall j$
\begin{equation*}
     \gamma w^H\mathbf{f}\left(\theta_j^{(n)}\right)w-1\ge 0  \quad \forall w \in \mathbb{C}^{d}, \quad w^Hw=1.
\end{equation*}
Since we are in the case where $w^H\mathbf{f}\left(\theta_j^{(n)}\right)w>0$, it is sufficient to choose $\gamma$ such that
\[\gamma>\gamma_1=\max_{\theta_j^{(n)}\neq \theta_0}\frac{1}{\min_{w\in \mathbb{C}^d, w^Hw=1}w^H\mathbf{f}\left(\theta_j^{(n)}\right)w}
    =\max_{\theta_j^{(n)}\neq \theta_0}\frac{1}{\lambda_{\min}\left(\mathbf{f}\left(\theta_j^{(n)}\right)\right)}.\]

For the case $\theta_j^{(n)} =\theta_0$, in order to prove (\ref{eq:dis_approx_prop_symb}), we have two sub-cases:
\begin{itemize}
    \item[a] if $w=q_{\bar{\jmath}}(\theta_0)$, then  $\gamma w^H\mathbf{f}(\theta_0)w-1+ |q_{\bar{\jmath}}(\theta_0)^Hw|^2=0$ and so (\ref{eq:dis_approx_prop_symb}) holds.
    \item[b] if $w\perp q_{\bar{\jmath}}(\theta_0)$, then $w^H\mathbf{f}(\theta_0)w>0$ and it is sufficient to choose 
    \[\gamma>\gamma_2\ge \frac{1}{\min_{w\in \mathbb{C}^d, w\perp q_{\bar{\jmath}}(\theta_0)}w^H\mathbf{f}(\theta_0)w}.\]
\end{itemize}
Integrating both scenarios, $\theta_j^{(n)}=\theta_0$ and $\theta_j^{(n)} \neq\theta_0$, we have successfully demonstrated that $\gamma=\max(\gamma_1,\gamma_2)$ fulfills the condition specified by equation (\ref{condW}).
\end{proof}
\end{theorem}

\begin{remark}
    In the preceding proof, we set the approximation property constant $\gamma$ to $\max(\gamma_1,\gamma_2)$, with $\gamma_1$ being dependent on the points $\theta_j^{(n)}$ and, consequently, on the matrix size $n$. This setup does not permit us to infer anything regarding the method's optimality. Nevertheless, by adopting a continuous rather than discrete perspective, we observe that the minimum eigenvalue of $\mathbf{f}\left(\theta\right)$ approaches zero if and only if $\theta$ tends to $\theta_0$. 
    
    Referring back to equation \eqref{eq:dis_approx_prop_symb}, for the case where $w=q_{\bar{\jmath}}(\theta_0)$, we can express it as:
    \[
    \lim_{\theta \rightarrow \theta_0} \gamma w^H\mathbf{f}\left(\theta\right)w-1+ |q_{\bar{\jmath}}(\theta_0)^Hw|^2 = 0
    \]
    which suggests that this term does not influence the selection of $\gamma$. Alternatively, for the condition where $w \perp q_{\bar{\jmath}}(\theta_0)$, we have:
    \[
    \lim_{\theta \rightarrow \theta_0} \gamma w^H\mathbf{f}\left(\theta\right)w-1+ |q_{\bar{\jmath}}(\theta_0)^Hw|^2 = \gamma w^H\mathbf{f}\left(\theta_0\right)w-1
    \]
    which is satisfied by choosing $\gamma$ to be greater than or equal to $\gamma_2$, hence confirming the selection of $\gamma=\max(\gamma_1,\gamma_2)$.

    Since this analysis does not depend on $n$, we can conclude that the approximation property constant can be chosen indepentent from the matrix size and hence the two-grid method has an optimal convergence rate.
\end{remark}

\subsection{Multigrid Convergence and Optimality}\label{ssec:v-cycle}
The aim of this subsection is to construct a multigrid procedure for block structured matrices that at the finest level consists of the grid transfer operator and post smoother that we analysed in the previous subsection. In particular, we consider the following V-cycle strategy:
\begin{enumerate}
\item At 1st level set $
  P_{n,k}^d=I_{n}\otimes q_{\bar{\jmath}}(\theta_0).$
 The system matrix at the coarse levels is 
$\mathcal {C}_{n_i}(\tilde{f}_1)$ associated with a scalar valued symbol $\tilde{f}_1$.

\item We can apply the TGM for the  $\mathcal {C}_{n_\ell}(\tilde{f}_{\ell})$  exploiting the convergence theory for scalar structured matrices \cite{ADS}. That is, for $\ell=1,\dots,\ell_{\min}$, choose $P_{n,k}^{(\ell)} = \mathcal{C}_{{n_{\ell}}}(p) K_{n_{\ell},k_{\ell}}$, where $p(\theta)$ is such that

\begin{equation*}
				\begin{split}
				|p(\theta)|^2+|p(\theta+\pi)|^2&>0, \quad \forall \, \theta\in [0,2\pi),\\
				\underset{\theta\to \theta_0}{\lim}\frac{|p(\theta+\pi)|^2}{\tilde{f}_{\ell}(\theta)}&<\infty.
				\end{split}
			\end{equation*}
   Moreover, the algorithm is completed by one step of scalar damped Jacobi as post smoother with appropriate relaxation parameter.
	\end{enumerate}	

By Theorems  \ref{thm:Rstub}, \ref{thm:smoothing_par}  and \ref{th:tgmopt}  for the finest level it holds 
\begin{equation}\label{eq:rho_finest}
    \rho({\rm TGM_0})\le {\sqrt{1-\frac{a_{0}}{\gamma_0}}}<1.
\end{equation}

In addition, exploiting the results in \cite{AD} concerning scalar generating functions, we obtain  for $\ell=1,\dots,\ell_{\min}$, 
					$$\rho({\rm TGM_{\ell}})\le{\sqrt{1-\frac{a_{\ell}}{\gamma_\ell}}}<1,$$
which implies that we have level independence.
				
	In \cite{ADS} the authors prove that in order to obtain the MGM optimal convergence, we need to prove 
\begin{equation}\label{eq:inf_min_vcycle_condition}
      \inf_{\ell_{\min}} \min_{0\le\ell\le \ell_{\min}} \frac{1}{a_\ell}>0.
\end{equation}

Furthermore, they show that the latter requirements hold when the minimum ranges over  $\ell=1,\dots, \ell_{\min}$, that is for scalar circulant matrices when applying the multigrid procedure of item 2.

The presented multigrid procedure is then optimally convergent since the fact that $a_0$ can be chosen different from zero is a direct consequence of relation (\ref{eq:rho_finest}). 

\subsection{Choice of Optimal Parameters}\label{ssec:optimal_parameters}

In Subsection \ref{ssec:smoothing_prop} we give necessary conditions, based on the symbol, on the relaxation parameter for the block Jacobi smoother such that the smoothing property is fulfilled. However, using similar techniques, in general it is not straightforward to find the optimal parameter in the range of admissible values so that the multigrid method converges in the fewest possible iterations. Yet, when we combine block Jacobi and the aggregation-based grid transfer operator, it is feasible to compute the symbol of the TGM iteration matrix
\begin{multline*}
M_n = \left(I_{nd}-\omega_{\rm{post}} \left(I_n \otimes \hat{\mathbf{f}}_0^{-1}\right)\mathcal{C}_{n}({\mathbf{f}})\right)
\left(I_{nd}- P^d_{n,k}\left((P^d_{n,k})^{H}
\mathcal{C}_{n}({\mathbf{f}}) P^d_{n,k}\right)^{-1}(P^d_{n,k})^{H}\mathcal{C}_{n}({\mathbf{f}})\right)\cdot\\
\left(I_{nd}-\omega_{\rm{pre}} \left(I_n \otimes \hat{\mathbf{f}}_0^{-1}\right)\mathcal{C}_{n}({\mathbf{f}})\right)
\end{multline*}
for solving a linear system with coefficient matrix $\mathcal{C}_{n}({\mathbf{f}})$, with $\mathbf{f}$ trigonometric polynomial defined as in Section \ref{sec:two_grid}, and with $P^d_{n,k}$ defined as in \eqref{eq:aggr_P_I}. Exploiting the circulant algebra and Lemma \ref{lem:coarse_symbol}, we can write the symbol $\mathbf{g}$ of the matrix $M_n$ as
\[
    \mathbf{g} = 
    \left(I_{d}-\omega_{\rm{post}} \hat{\mathbf{f}}_0^{-1}{\mathbf{f}}\right)
    \left(I_d-
    \frac{1}{\tilde{f}}\,q_{\bar{\jmath}}(\theta_0)q_{\bar{\jmath}}^H(\theta_0)
    \mathbf{f}\right)
    \left(I_{d}-\omega_{\rm{pre}} \hat{\mathbf{f}}_0^{-1}{\mathbf{f}}\right)
\]
with $\tilde{f}$ defined as in \eqref{eq:f_tilde}.

The anlysis in Subsection \ref{ssec:approximation_prop} guarantees that taking a uniform sampling of the function $\mathbf{g}$ at the points where it is defined, computing the eigenvalues of each sample and taking the maximum of the computed values is a reasonable approximation of the spectral radius of the iteration matrix. We exploit this reasoning in Subsection \ref{ssec:over_relaxation}.

\section{Numerical Experiments}\label{sec:num_experiments}

The present section is devoted to show numerically the efficiency of several multigrid strategies obtained exploiting the theoretical results of Sections \ref{ssec:gtoperator_and_symbol} and \ref{ssec:smoothing_prop}. In particular, in Subsection  \ref{ssec:scalar_vs_block} we show that the use of block Jacobi smoothers improves the performance of the exiting symbol-based multigrid procedure involving classical grid transfer operator with respect to scalar smoothers. Subsections \ref{ssec:aggregation_experiments} and \ref{ssec:over_relaxation} are focused on testing the grid transfer operator defined by (\ref{eq:aggr_P_I}) and proper over-relaxation strategies. Finally, we compare the results of the two procedures both as standalone methods and as preconditioner for Krylov iterative methods in Subsection \ref{ssec:PPCG}. 
Most of the examples considered in the numerical section  are  block circulant and block Toeplitz(-like)  linear systems stemming from the discretization with $\mathbb{Q}_{d} $ Lagrangian FEM approximation of a second order differential problem and B-spline discretization with non-maximal regularity. However, we also consider an “artificial” block Toeplitz matrix constructed starting by a scalar Toeplitz by manipulating the associated generating function.
Indeed, this toy case shows in a immediate way how the conditions outlined in Theorem \ref{thm:smoothing_par}  possess a simplified expression in some practical cases.  In the implementation we use the standard stopping criterion  $\frac{\|r^{(k)}\|_2}{\|\mathbf{b}\|_2}<\epsilon$, where $r^{(k)}=\mathbf{b}-A\mathbf{x}^{(k)}$ and tolerance $\epsilon=10^{-6}$. We consider the right-hand side $\mathbf{b}$ defined as $\mathbf{b} = A\mathbf{x}$ and we take the null initial guess. All the tests are performed using MATLAB 2022b and the error equation at the coarsest level is solved with the MATLAB backslash function.   In our context we stop the recursion of the V-cycle when the matrix size is smaller than 64.

\subsection{Examples}\label{ssec:example_descr}
In the first part of the numerical section we introduce all the tested linear systems and we collect all the relevant spectral information on the associated symbols.

\subsubsection{Scalar Toeplitz Matrices Interpreted as Block Toeplitz Matrices}
\label{ssec:scalartoblock}

The first example we consider is that of an ``artificial'' block Toeplitz matrix constructed starting by a scalar Toeplitz.
Indeed, given a univariate and scalar-valued generating function $f(\theta)$ it is possible to compute the corresponding $d\times d$ matrix-valued generating function $\mathbf{f}^{[d]}$ defined by

  \begin{align}
\mathbf{f}^{[d]}(\theta)=\sum_{\ell=-r}^r \underbrace{T_d({\rm e}^{-\mathbf{i}\ell d\theta}f(\theta))}_{\hat{\mathbf{f}}_\ell^{[d]}}{\rm e}^{\mathbf{i}\ell\theta},\label{eq:blockversionf}
 \end{align} 

where $\hat{\mathbf{f}}_\ell^{[d]}$ are the corresponding matrix-valued Fourier coefficients. Then,

  \begin{align}
T_{ns}(f)=T_n(\mathbf{f}^{[d]}).\nonumber
 \end{align} 
In particular we consider the example in \cite{huckle} where $f(\theta)=2-2\cos\theta$ and the associated matrix-valued version is $\mathbf{f}^{[d]}(\theta)= \hat{\mathbf{f}}_0^{[d]}+\hat{\mathbf{f}}_
{-1}^{[d]}{\rm e}^{-\mathbf{i}\theta}+\hat{\mathbf{f}}_1^{[d]}{\rm e}^{\mathbf{i}\theta}$,
where
\[\hat{\mathbf{f}}_0^{[d]}=T_d(2-2\cos\theta); \quad \hat{\mathbf{f}}_{-1}^{[d]}=-e_de_1^T; \quad \hat{\mathbf{f}}_1^{[d]}=(\hat{\mathbf{f}}_{-1}^{[d]})^T=-e_1e_d^T.\]
The matrix valued function $\mathbf{f}^{[d]}(\theta)$ is such that 

\begin{itemize}
		\item $\lambda_{1}\left(\mathbf{f}^{[d]}(\theta)\right)$ has a zero of order 2 in $\theta_0=0$.
		\item $\mathbf{f}^{[d]}(0)q_{1}(0)= 0$, $q_{1}(0)=[1,\dots,1]^T$.
		\item The Fourier coefficients $\hat{\mathbf{f}}_1^{[d]}$ and $\hat{\mathbf{f}}_{-1}^{[d]}$ enjoy the expression of Theorem \ref{th:simplifyrank2smoother} with $u=-e_d$, and $w=e_1$. Since $e_1=Ye_d$ and $e_d=Ye_1$, with $Y$  backward identity matrix of size $d$, we obtain
\begin{equation*}
\max_\theta \lambda_{\max}\left(\left(\hat{\mathbf{f}}_0^{[d]}\right)^{-\frac{1}{2}}\mathbf{f}^{[d]}\left(\hat{\mathbf{f}}_0^{[d]}\right)^{-\frac{1}{2}}\right)= 1+ \max_\theta \left(v^TYv\cos\theta+ \sqrt{(v^TYv)^2(\cos^2\theta-1)+\|v\|^4}\right)
\end{equation*}
with $v= \left(\hat{\mathbf{f}}_0^{[d]}\right)^{-\frac{1}{2}} e_d$. The latter permits to compute the quantity

$$\left\|\left(\hat{\mathbf{f}}_0^{[d]}\right)^{-\frac{1}{2}}\mathbf{f}^{[d]}\left(\hat{\mathbf{f}}_0^{[d]}\right)^{-\frac{1}{2}}\right\|_{\infty}=2.$$
\item The function $\tilde{f}(\theta) = [1,\dots,1] \mathbf{f}^{[d]}(\theta) [1,\dots,1]^T $ is the scalar valued function $\tilde{f}(\theta)= 2-2\cos{\theta}$ for all $d$.
		\end{itemize}

Exploiting the computation in (\ref{eq:blockversionf}),  we can also compute the circulant matrix $C_n(\mathbf{f}^{[d]})$ generated by $\mathbf{f}^{[d]}$.

\subsubsection{Stiffness Matrices Using $\mathbb{Q}_{\textcolor{black}{d}}$ Lagrangian FEM}
\label{ssec:FEM}
The second case we present is given by the classical block structured problem stemming from the $\mathbb{Q}_{d} $ Lagrangian FEM approximation of  a second order differential problem.

The 1D problem {is given by:}
Find $u$ such that
\begin{small}
\begin{equation}\label{FEM_problem}
\begin{cases}
&-u''(x)=\psi(x) \quad {\rm on}\, \, (0,1), \\
& u(0)=u(1)=0 ,
\end{cases}
\end{equation}
\end{small}where $\psi(x)\in L^2\left(0,1\right)$. 

 If we discretize (\ref{FEM_problem}) using ${\mathbb{Q}_2}$ Lagrangian FEM \cite{qp}, the scaled stiffness matrix is a rank 1 correction of the Toeplitz matrix $T_n(\mathbf{\textbf{f}}_{\mathbb{Q}_2}),$ generated by the function  
\begin{equation*}
\begin{split}
\mathbf{f}_{\mathbb{Q}_2}(\theta)=&\frac{1}{3}\left(\begin{bmatrix}
16 & -8 \\
-8 &14
\end{bmatrix}+ \begin{bmatrix}
0 & -8 \\
0 &1
\end{bmatrix}{\rm e}^{\iota\theta}+\begin{bmatrix}
0 & 0\\
-8 &1
\end{bmatrix} {\rm e}^{-{\iota}				\theta}\right)=\\&\frac{1}{3}\begin{bmatrix}
16 & -8(1+{\rm e}^{\iota \theta})\\
-8(1+{\rm e}^{-\iota \theta})& 14+ {\rm e}^{\iota \theta}+{\rm e}^{-\iota \theta}\end{bmatrix}.
\end{split}
\end{equation*}
The function $\mathbf{f}_{\mathbb{Q}_2}(\theta)$ possesses the following properties:
		\begin{itemize}
		\item $\lambda_{1}(\mathbf{\textbf{f}}_{\mathbb{Q}_2}(\theta))$ has a zero of order 2 in $\theta_0=0$.
		\item $\mathbf{\textbf{f}}_{\mathbb{Q}_{{2}}}(0) q_{1}(0)= 0$, $q_{1}(0)=[1,1]^T$.
		\item The Fourier coefficients $\hat{\mathbf{f}}_{1}$ and $\hat{\mathbf{f}}_{-1}$ enjoy the expression of those of Theorem \ref{th:simplifyrank2smoother} and we can compute the quantity $\left\|\hat{\mathbf{f}}_0^{-\frac{1}{2}}\mathbf{\textbf{f}}_{\mathbb{Q}_2}\hat{\mathbf{f}}_0^{-\frac{1}{2}}\right\|_{\infty}=2$.
  \item The function $\tilde{f}_{\mathbb{Q}_2}(\theta)= [1,1] \mathbf{f}_{\mathbb{Q}_2}(\theta) [1,1]^T $ is the scalar valued function $\tilde{f}_{\mathbb{Q}_2}(\theta)= c_2 ( 2-2\cos{\theta})$, with $c_2$ constant $c_2 =\frac{1}{2} [1,1] \hat{\mathbf{f}}_0[1,1]^T.$
		\end{itemize}

Moreover these properties are valid for a general degree $d$ and related symbols $\mathbf{f}_{\mathbb{Q}_d}(\theta)$.
If in (\ref{FEM_problem}) periodic boundary conditions are imposed, then the stiffness matrix is the $dn\times dn$ circulant matrix $C_n(\mathbf{\textbf{f}}_{\mathbb{Q}_d})$ generated by $\mathbf{\textbf{f}}_{\mathbb{Q}_d}$.

\subsubsection{Stiffness matrices Using B-Splines}
\label{ssec:Splines}

The last example is given by the structured linear systems obtained when using B-Spline discretization of the problem (\ref{FEM_problem}).  In \cite{MR4019290} the authors analyse the  matrix-valued function associated with  B-Spline approximation for different values of degree $p$ and regularity $k$. Here we only consider the case $A_{p,k}$ for the pairs
$(p,k)$ equal to $(2,0)$, $(3,1)$ and $(3,0)$. Matrices $A_{p,k}$ are low-rank correction of the Toeplitz matrices generated by the following functions
\begin{align}
\label{eq:symb_bspline}
&\mathbf{f}^{(2,0)}(\theta)=\frac{1}{3}\left(\begin{bmatrix}
4 & -2 \\
-2 &8
\end{bmatrix}+ \begin{bmatrix}
0 & -2 \\
0 &-2
\end{bmatrix}{\rm e}^{\iota\theta}+\begin{bmatrix}
0 & 0\\
-2&-2
\end{bmatrix} {\rm e}^{-{\iota}\theta}\right)\\
&\mathbf{f}^{(3,1)}(\theta)=\frac{1}{40}\left(\begin{bmatrix}
48 & 0 \\
0&48
\end{bmatrix}+ \begin{bmatrix}
-15& -15 \\
-3 &-15
\end{bmatrix}{\rm e}^{\iota\theta}+\begin{bmatrix}
-15 & -3\\
-15 &-15
\end{bmatrix} {\rm e}^{-{\iota}\theta}\right)\\
&\mathbf{f}^{(3,0)}(\theta)=\frac{1}{10}\left(\begin{bmatrix}
12 & 3 &-6\\
3 &12 &-9\\
-6 &-9 &36\\
\end{bmatrix}+ \begin{bmatrix}
0& 0 &-9\\
0 &0 &-6\\
0&0 &3\\
\end{bmatrix}{\rm e}^{\iota\theta}+\begin{bmatrix}
0 & 0 &0\\
0 &0 &0\\
-9 &-6 &3\\
\end{bmatrix} {\rm e}^{-{\iota}\theta}\right).
\end{align}
The latter verify the following:
\begin{itemize}
    \item all the three minimum eigenvalue functions of $\mathbf{f}^{(2,0)}$, $\mathbf{f}^{(3,0)}$ and $\mathbf{f}^{(3,1)}$ have  a zero of order 2 associated with the eigenvector of all ones. 
    \item For all the 3 cases is possible to compute easily the quantities $$\left\|\left(\hat{\mathbf{f}}_0^{(p,k)}\right)^{-\frac{1}{2}}\mathbf{\textbf{f}}^{(p,k)}\left(\hat{\mathbf{f}}_0^{(p,k)}\right)^{-\frac{1}{2}}\right\|_{\infty}=2,$$ either exploiting Theorem \ref{th:simplifyrank2smoother} and the structure of the Fourier coefficient of $\mathbf{f}^{(2,0)}$ and $\mathbf{f}^{(3,0)}$ or the diagonal expression of the 0-th Fourier coefficient of $\mathbf{f}^{(3,1)}$. 
    \item  All the generating functions at coarser levels $\tilde{f}^{(p,k)} = [1,\dots,1] \mathbf{f}^{(p,k)}[1,\dots,1]^T $ have the form $\tilde{f}^{(p,k)} (\theta)= c (2-2\cos{\theta})$, with $c =\frac{1}{2} [1,\dots,1]  \hat{\mathbf{f}}_0^{(p,k)}[1,\dots,1]^T.$
\end{itemize}

\subsection{Scalar and Block smoothers in matrix-valued multigrid approach}\label{ssec:scalar_vs_block}
 Two multigrid optimal strategies, such as the geometric projection operator and the standard bisection grid transfer operator, were already studied in \cite{MR4284081,FRTBS} for the linear systems  involved in Section \ref{ssec:FEM} for the FEM discretization with $d=2$, $d=4$ and $d=8$.   However, in this subsection we compare the efficiency of matrix-valued multigrid approach when using  scalar and block Jacobi methods as smoothers. 

We construct a matrix-valued multigrid method for the block circulant matrix $C_{n}(\mathbf{\textbf{f}}_{\mathbb{Q}_d})$, using a grid trnasfer operator   \begin{equation*}
P_{n,k}^{_{\mathbb{Q}_{d}}}=C_n(\mathbf{p}_{_{\mathbb{Q}_{d}  }})(K_{n,k} \otimes I_d),
\end{equation*}
where $\mathbf{p}_{_{\mathbb{Q}_{d}  }}=\hat{a}_0+\hat{a}_{-1} {\rm e}^{-\imath\theta}+\hat{a}_{1} {\rm e}^{\imath\theta}$ is a trigonometric polynomial. The coefficients  $\hat{a}_0,\hat{a}_{-1}, \hat{a}_{1} $ depends on whether $d$ is even or odd. For the general expressions for all $d$ see \cite{MR4389580}[Section 5.2].

As scalar pre/post smoother we consider one step of relaxed Jacobi method with iteration matrix equal to $V_{n}:=I_{N}-\omega D_N^{-1}T_{n}(\mathbf{\textbf{f}}_{\mathbb{Q}_d})$, where ${D}_n:= \min_{j=1,\dots,\textcolor{black}{d}}{\left(\hat{\mathbf{f}}_{0}\right)_{(j,j)}}I_N$. The range of admissible values  $\omega$ verifies the following inequality:
\begin{equation}\label{eqn:jacobi_omega}
0<\omega< \frac{2}{\left\|\left(\diag \hat{\mathbf{f}}_0\right)^{-\frac{1}{2}}\mathbf{{f}}_{\mathbb{Q}_d}\left(\diag \hat{\mathbf{f}}_0\right)^{-\frac{1}{2}}\right\|_{\infty}}.
\end{equation}

Precisely we choose $\omega_{\rm post}$ as the midpoint of the interval of admissible values and $\omega_{\rm pre}=\frac{3}{2}\omega_{\rm post}$.

Concerning the block-smoothing strategy, we already prove that  an iteration matrix of the form (\ref{eq:defVn_block_Jacobi}) yields to the admissible smoothing parameters 

\[  0<\omega< \frac{2}{\left\|\hat{\mathbf{f}}_0^{-\frac{1}{2}}\mathbf{{f}}_{\mathbb{Q}_d}\hat{\mathbf{f}}_0^{-\frac{1}{2}}\right\|_{\infty}},\]

with  $\left\|\hat{\mathbf{f}}_0^{-\frac{1}{2}}\mathbf{\textbf{f}}_{\mathbb{Q}_d}\hat{\mathbf{f}}_0^{-\frac{1}{2}}\right\|_{\infty}=2$ for any $d$. Then we fix for the experiments  $\omega_{\rm pre}=3/4$ and $\omega_{\rm post}=1/2$ to dampen different frequencies of the error.

Table \ref{tab:simax_blockvsscalar_jacobi} presents a comparison between the block and scalar Jacobi smoothers performance in terms of iterations needed for convergence and computational solving times when applying the V-cycle to  $C_n(\mathbf{{f}}_{\mathbb{Q}_d})$. Both methods show a convergent and optimal behaviour in term of iterations. However, the block smoother approach is preferable when increasing the size of the block $d$. Indeed, even if the cost of the system for block Jacobi is naturally high for bigger $d$, total solving time remains lower, since the number of iterations required of the global method remains equal or decreases for $d=2,4,8$. 

\begin{footnotesize}
\begin{table}[htb]\begin{center}
		\begin{tabular}{c|c|cc|cc|cc|cc}
	&& \multicolumn{2}{c}{{$d=2$}}&\multicolumn{2}{c}{{$d=3$}}&\multicolumn{2}{c}{{$d=4$}}&\multicolumn{2}{c}{$d=8$}\\
			\hline
			
& $t $&  {T(s)} & {Iter} & {T(s)} & {Iter} & {T(s)} & {Iter} & {T(s)} & {Iter} \\
\hline\hline
\multirow{5}{*}{Block Jacobi} &			10&0.0043  & 8 &0.0066&8& 0.0079&7&0.0241&6\\
\cline{2-2}	&		                      			11&0.0079  & 8 &0.0119&8& 0.0156&7&0.0406&6\\
\cline{2-2}	&		                     			12&0.0163  & 8 &0.0257&8& 0.0302&7&0.0800&6\\
\cline{2-2}	&		                     			13&0.0295  & 8 &0.0501&8& 0.0595&7&0.1413&6\\
\cline{2-2}	&		                     			14&0.0573  & 8 &0.0960&8& 0.1167&7&0.2853&6\\
\hline\hline	
\multirow{5}{*}{Scalar Jacobi}	&		10& 0.0016& 7 &0.0035&11&0.0069 &17&0.4138&251\\
\cline{2-2}								&		11& 0.0023& 7 &0.0055&11&0.0140 &17&0.5709&253\\
\cline{2-2}								&		12& 0.0055& 7 &0.0120&11& 0.0248&17&0.9282&251\\
\cline{2-2}								&		13& 0.0086& 7 &0.0204&11&0.0426 &17&1.5912&252\\
\cline{2-2}								&		14& 0.0170& 7 &0.0386&11&0.0808 &17&2.9187&252\\
			\hline
		\end{tabular}
	\end{center}
	
	\caption{Comparison between the block and scalar Jacobi smoothers performance in terms of iterations and CPU time for the V-cycle applied to  $C_n(\mathbf{\textbf{f}}_{\mathbb{Q}_d})$. }\label{tab:simax_blockvsscalar_jacobi}

\end{table}
\end{footnotesize}

\subsection{Symbol-based aggregated multigrid methods}\label{ssec:aggregation_experiments}
In this subsection we numerically verify the results that we proved in Section \ref{sec:proofs}. The grid transfer operator can be seen in Equation \eqref{eq:aggr_P_I} and we use it in combination with the block Jacobi smoother. The smoothing parameters are chosen according to Theorem \ref{thm:smoothing_par}, in such a way that the smoothing properties are fulfilled. In particular, we make the same choices as in Subsection \ref{ssec:scalar_vs_block}. Concerning the grid transfer operator, in Subsections \ref{ssec:scalartoblock}--\ref{ssec:Splines} we observed that in all the examples that we are considering, the eigenvector of the generating function associated to the null eigenvalue is the vector of all ones $[1,\dots,1]^T$. Therefore, the grid transfer operator defined in \eqref{eq:aggr_P_I} becomes
\[
P^d_{n,k}=I_{n}\otimes \begin{bmatrix}
    1 \\ \vdots \\ 1
\end{bmatrix}_{d \times 1}.
\]
Furthermore, all the generating functions at coarser levels $\tilde{f} = [1,\dots,1] \mathbf{f} [1,\dots,1]^T $ have the form $\tilde{f}(\theta)= c (2-2\cos{\theta})$, with $c$ constant depending on the $0$-th Fourier coefficient. So, when we apply the V-cycle strategy described in Subsection \ref{ssec:v-cycle}, we can take the standard linear interpolation as grid transfer operator and $\omega_{\rm post}=1/2$, see \cite{AD}.

The following tables show the optimality of the aggregated based multigrid strategy in combination with Jacobi as post smoother.  In particular we consider one iteration of block Jacobi smoother at the finest level with parameter $\omega_{\rm post}= 1/2$ as the midpoint of the interval of admissible values given Theorem \ref{thm:smoothing_par} and the properties of the involved generating functions listed in Subsecion \ref{ssec:example_descr}. 

We report the number of iterations of the TGM and V-cycle methods when applied to matrix system $C_n(\mathbf{f}^{[d]})$, varying $d=2,4,8$ in Table \ref{tab:TGM_Vcycle_BJ_C_toy}. In this subsection, we are not focused on the choice of the optimal parameters, our goal is instead to numerically validate the results in Section \ref{sec:proofs}, which guarantee that the number of multigrid iterations for convergence does not depend on the matrix size.


\begin{table}[htb]
\begin{center}
\begin{tabular}{c c c c | c c c c c}
\hline
\multirow{2}{*}{$t$} & \multicolumn{3}{c}{$\text{TGM}$} & \multicolumn{3}{c}{$\text{V-cycle}$} \\
\cline{2-4} \cline{5-7}
& $d=2$ & $d=4$ & $d=8$ & $d=2$ & $d=4$ & $d=8$ \\
\hline
15 & 33 & 52 & 88 & 42 & 69 & 115 \\
16 & 33 & 52 & 88 & 42 & 69 & 115 \\
17 & 33 & 52 & 88 & 42 & 69 & 115 \\
18 & 33 & 52 & 88 & 42 & 69 & 115 \\
19 & 33 & 52 & 88 & 42 & 69 & 115 \\
20 & 33 & 52 & 88 & 42 & 69 & 115 \\
\hline
\end{tabular}
\end{center}
\caption{Number of iterations for the TGM and V-cycle methods applied to $C_n(\mathbf{f}^{[d]})$ varying the block size $d$. Only 1 iteration of block Jacobi post smoother is applied.}
\label{tab:TGM_Vcycle_BJ_C_toy}
\end{table}


An analogous behaviour is shown in Table \ref{tab:TGM_Vcycle_BJ_C_Fem} for TGM and V-cycle iterations when increasing the matrix size of the matrices $C_n(\mathbf{\textbf{f}}_{\mathbb{Q}_d})$ for $d=2,4,8$.


\begin{table}[htb]
\begin{center}
\begin{tabular}{c c c c | c c c c c}
\hline
\multirow{2}{*}{$t$} & \multicolumn{3}{c}{$\text{TGM}$} & \multicolumn{3}{c}{$\text{V-cycle}$} \\
\cline{2-4} \cline{5-7}
& $d=2$ & $d=4$ & $d=8$ & $d=2$ & $d=4$ & $d=8$ \\
\hline
15 & 37 & 64 & 121 & 48 & 84 & 155 \\
16 & 37 & 64 & 121 & 48 & 84 & 155 \\
17 & 37 & 64 & 121 & 48 & 84 & 155 \\
18 & 37 & 64 & 121 & 48 & 84 & 155 \\
19 & 37 & 64 & 121 & 48 & 84 & 155 \\
20 & 37 & 64 & 121 & 48 & 84 & 155 \\
\hline
\end{tabular}
\end{center}
\caption{Number of iterations for the TGM and V-cycle methods applied to $C_n(\mathbf{\textbf{f}}_{\mathbb{Q}_d})$ varying the block size $d$. Only 1 iteration of block Jacobi post smoother is applied.}
\label{tab:TGM_Vcycle_BJ_C_Fem}
\end{table}

Finally, TGM and V-cycle aggregated based methods are applied on the circulant matrices $\mathcal{C}_n(\mathbf{\textbf{f}}^{(p,k)})$ obtained in B-spline approximation for the pairs $(p,k)$ equal to $(2,0)$, $(3,1)$ and $(3,0)$. 
We report in Table \ref{tab:TGM_Vcycle_BJ_C_Bspline}  the number of iterations needed to reach the desired tolerance for the TGM and V-cycle methods, respectively. In both cases, we observe that the number of iterations remains constant when increasing the matrix size.


\begin{table}[htb]
\begin{center}
\begin{tabular}{c c c c | c c c c c}
\hline
\multirow{2}{*}{$t$} & \multicolumn{3}{c}{$\text{TGM}$} & \multicolumn{3}{c}{$\text{V-cycle}$} \\
\cline{2-4} \cline{5-7}
& $(2,0)$&$(3,1)$&$(3,0)$ & $(2,0)$&$(3,1)$&$(3,0)$ \\
\hline
15 & 24 & 32 & 30 & 29 & 34 & 38 \\
16 & 24 & 32 & 30 & 29 & 34 & 38 \\
17 & 24 & 32 & 30 & 29 & 34 & 38 \\
18 & 24 & 32 & 30 & 29 & 34 & 38 \\
19 & 24 & 32 & 30 & 29 & 34 & 38 \\
20 & 24 & 32 & 30 & 29 & 34 & 38 \\
\hline
\end{tabular}
\end{center}
\caption{Number of iterations for the TGM and V-cycle methods when applied to $\mathcal{C}_n(\mathbf{\textbf{f}}^{(p,k)})$ for $(p,k)$ equal to $(2,0)$, $(3,1)$ and $(3,0)$.  Only 1 iteration of block Jacobi post smoother is applied.}
\label{tab:TGM_Vcycle_BJ_C_Bspline}
\end{table}

\subsection{Over-relaxation}\label{ssec:over_relaxation}

Efficient convergence of aggregation--based multigrid methods, especially for elliptic problems of the second order, can be substantially enhanced through strategic over-relaxation of the coarse grid correction, a concept explored by Braess in \cite{MR1370108}. Figure 3 in Braess's work shows this with a function linear on three segments, demonstrating that the approximation properties can be less than ideal, particularly evident in the one-dimensional Poisson equation scenario. The sub-optimal approximation quality is enhanced by introducing in the coarse-grid correction an over-relaxation factor, $\alpha > 1$, where numerical calculations indicate $\alpha \approx 2$ to be optimal for the Poisson equation, although \(\alpha = 1.8\) is chosen to prevent overshooting. Further explanations related to this approach are present in \cite{MR2385900,MR1820878}.

The iteration matrix of the TGM with over-relaxation of the coarse grid correction is
\[{\rm TGM}(A_{n},V_{n,\rm{pre}},V_{n,\rm{post}},P_{n,k})=
V_{n,\rm{post}}
\left[I_{n}-\alpha P_{n,k}\left(P_{n,k}^{H}
A_{n}P_{n,k}\right)^{-1}P_{n,k}^{H}A_{n}\right]V_{n,\rm{pre}},\]
This corresponds to perform an over-relaxation when computing the interpolation of the error $$\tilde{y}_k= \alpha P_{n,k} y_k.$$
We perform exactly one iteration of pre and post smoother and, for simplicity, we consider the same value of pre and post smoothing parameter $\omega$. 
 The choice of $\omega$ can be improved in relation to $\alpha$. Precisely, we select the pair $(\alpha,\omega)$ which minimizes the spectral radius of the iteration matrix.

In the following, we apply this strategy to the examples  of Subsection \ref{ssec:scalartoblock}, Subsection \ref{ssec:FEM}, and Subsection \ref{ssec:Splines}. 
An efficient computation of the pseudo-spectral radius \cite{MR2369301} can be performed considering the associated circulant case, for which we can easily compute the maximum of the eigenvalue functions of the spectral symbol of the TGM iteration matrix, as we explained in Subsection \ref{ssec:optimal_parameters}. Indeed, the symbol of the TGM iteration matrix can be computed in analytic form also in the over-relaxed scenario, with a dependence on $\alpha$. Moreover, its eigenvalue functions can either be computed analytically or evaluated on a uniform grid on $[0,2\pi]$. These computations allow us to choose the best pair of smoothing and over-relaxation parameters in the circulant case. The presence of possible outliers for the  TGM iteration matrix in the Toeplitz case can alter the choice of the best pair $(\alpha_{\rm opt},\omega_{\rm opt})$ when using the multigrid as a standalone method. Yet, we show that the optimal value estimated in the circulant case does not differ too much from the optimal one.

Figure \ref{fig:pair_2meno2} shows the magnitude of the spectral radius of the TGM iteration matrix for the system  $C_n(\mathbf{f}^{[2]})$ computed over a range of admissible values for $\omega$ and $\alpha$. Precisely we consider $17$ equispaced values in the interval $[0.5,0.9]$ for $\omega$ and $11$ equispaced values in the interval $[1,3]$ for $\alpha$.

\begin{figure}[h!]
\centering
\includegraphics[width=0.70\textwidth]{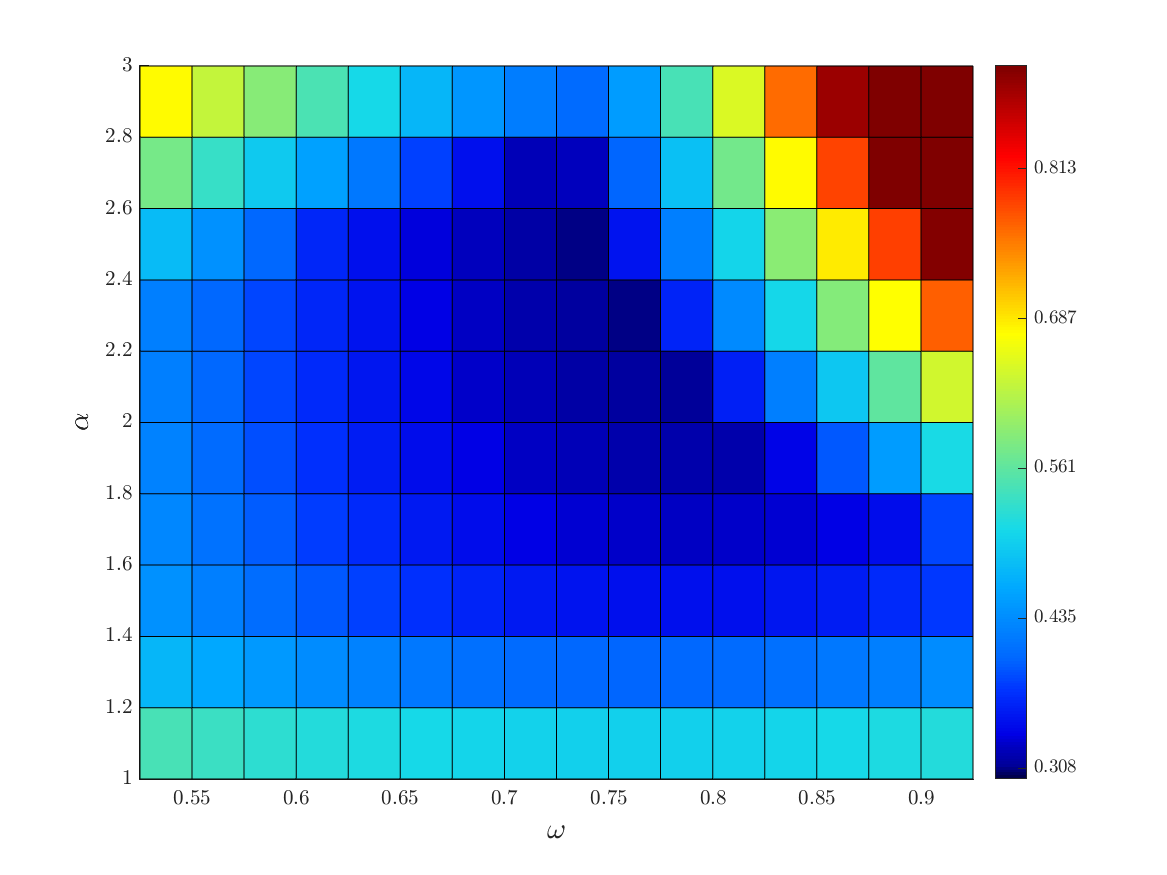}
\caption{Plot of the magnitude of the spectral radius of the TGM iteration matrix for the system  $C_n(\mathbf{f}^{[2]})$ computed over equispaced values of the pair $(\alpha,\omega)$. }\label{fig:pair_2meno2}
\end{figure}

In Table \ref{tab:pair_2meno2} we show the results in term of iterations with the estimated choices $\alpha_{\rm est}=2.2$ and $\omega_{\rm est}=0.75$ which provide $\rho(TGM)\approx 0.308$ in comparison with the case $\alpha=1$ (without over-relaxation) and the value $\omega=0.75$ in $[0,1]$ in which $\rho(TGM)\approx 0.5$.

\begin{table}
\begin{center}
		\begin{tabular}{c|cc|cc}
			
		$N=2\cdot 2^t $	& \multicolumn{2}{c|}{{$(\alpha_{\rm est},\omega_{\rm est})=(2.2, 0.75)$}} & \multicolumn{2}{c}{{$(\alpha=1,\omega=0.75)$}} \\
			\hline
			{$t$} & {TGM} & {V-Cycle} & {TGM} & {V-Cycle} \\
			\hline
		    8 & 11 & 11 & 14 & 16 \\
			9 & 11 & 11 & 14 & 16 \\
			10 & 11 & 11 & 14 & 16\\
			11 & 11 & 11 & 14 & 16  \\
			12& 11 & 11 &14 & 16\\
		
	\end{tabular}
		\caption{ Two-grid and V-cycle iterations with and without the over-relaxation strategy for the matrix $C_n(\mathbf{f}^{[2]})$.} \label{tab:pair_2meno2}
		\end{center}
\end{table}

\begin{table}
\begin{center}
		\begin{tabular}{c|cc|cc|cc}			
		$N=2\cdot 2^t $	& \multicolumn{2}{c|}{{$(\alpha_{\rm est},\omega_{\rm est})$}} 	& \multicolumn{2}{c|}{{$(\alpha_{\rm opt},\omega_{\rm opt})$}} & \multicolumn{2}{c}{{$(\alpha=1,\omega=0.75)$}} \\
			\hline
			{$t$} & {TGM} & {V-Cycle} & {TGM} & {V-Cycle} & {TGM} & {V-Cycle} \\
			\hline
		    8 & 12 & 11 & 10 & 11 & 13 & 15 \\
			9 & 11 & 11 & 11 & 11 & 14 & 16 \\
			10 & 11 & 11 & 10 & 11 & 13 & 15\\
			11 & 11 & 11 & 11 & 11 & 14 & 16  \\
			12& 11 & 11 &10 & 11 & 14 & 16\\
                13& 11 & 11 &11 & 11 & 13 & 16\\	
	\end{tabular}
		\caption{ Two-grid and V-cycle iterations with and without the over-relaxation strategy for the matrix $T_n(\mathbf{f}^{[2]})$. In this case, $(\alpha_{\rm est},\omega_{\rm est})=(2.6,0.725)$ and $(\alpha_{\rm opt},\omega_{\rm opt})=(1.8,0.775)$.} \label{tab:pair_2meno2_toep}
		\end{center}
\end{table}

Figure \ref{fig:pair_FEM} shows the magnitude of the spectral radius of the TGM iteration matrix for the system  $C_n(\mathbf{\textbf{f}}_{\mathbb{Q}_2})$ computed over a range of admissible values for $\omega$ and $\alpha$. Precisely we consider $13$ equispaced values in $[0.5,0.9]$ for $\omega$ and $16$ equispaced values in $[1,3.4]$ for $\alpha$.

\begin{figure}[htb]
\centering
\includegraphics[width=0.70\textwidth]{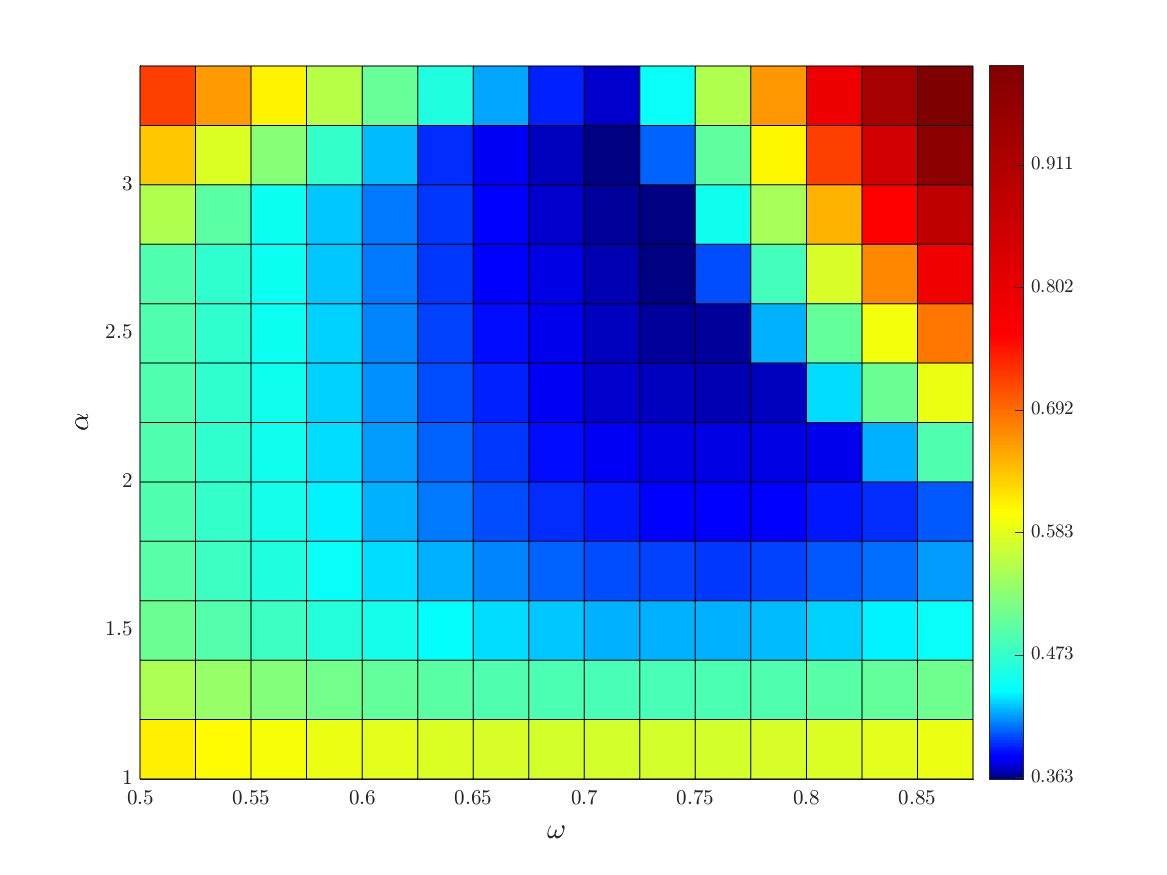}
\caption{Plot of the magnitude of the spectral radius of the TGM iteration matrix for the system  $C_n(\mathbf{\textbf{f}}_{\mathbb{Q}_2})$ computed over 15 equispaced values of the pair $(\alpha,\omega)$. }\label{fig:pair_FEM}
\end{figure}

In Table \ref{tab:pair_FEM} we show the results in term of iterations with the estimated choices $\alpha_{\rm est}=2.6$ and $\omega_{\rm est}=0.725$ which provide $\rho(TGM)\approx 0.363$ in comparison with the case $\alpha=1$ (without over-relaxation)  and $\omega=0.725$ which provide $\rho(TGM)\approx 0.571$.

\begin{table}
\begin{center}
		\begin{tabular}{c|cc|cc}
			
		$N=2\cdot 2^t $	& \multicolumn{2}{c|}{{$(\alpha_{\rm est},\omega_{\rm est})=(2.2,0.75)$}} & \multicolumn{2}{c}{{$(\alpha=1,\omega=0.725)$}} \\
			\hline
			{$t$} & {TGM} & {V-Cycle} & {TGM} & {V-Cycle} \\
			\hline
			8 & 12 & 12 & 16 & 18  \\
			9 &  12& 12 &  16 & 19 \\
			10 &  12& 12 & 16 & 18\\
			11 &  12& 12 & 16 & 18  \\
			12&  12& 12 &16 & 18 \\
			13 & 12 & 12  & 16& 18\\
		
	\end{tabular}
		\caption{ Two-grid and V-cycle iterations with and without the over-relaxation strategy for the matrix $C_n(\mathbf{{f}}_{\mathbb{Q}_2})$.}
		 \label{tab:pair_FEM}
		\end{center}
\end{table}

\begin{table}
\begin{center}
		\begin{tabular}{c|cc|cc|cc}			
		$N=2\cdot 2^t $	& \multicolumn{2}{c|}{{$(\alpha_{\rm est},\omega_{\rm est})$}} 	& \multicolumn{2}{c|}{{$(\alpha_{\rm opt},\omega_{\rm opt})$}} & \multicolumn{2}{c}{{$(\alpha=1,\omega=0.725)$}} \\
			\hline
			{$t$} & {TGM} & {V-Cycle} & {TGM} & {V-Cycle} & {TGM} & {V-Cycle} \\
			\hline
		    8 & 15 & 12 & 12 & 12 & 16 & 18 \\
			9 & 14 & 12 & 12 & 12 & 16 & 19 \\
			10 & 14 & 12 & 12 & 12 & 16 & 18\\
			11 & 13 & 12 & 12 & 12 & 16 & 18  \\
			12& 13 & 12 &12 & 12 & 16 & 18\\
                13& 13 & 12 &12 & 12 & 16 & 18\\	
	\end{tabular}
		\caption{ Two-grid and V-cycle iterations with and without the over-relaxation strategy for the matrix $T_n(\mathbf{{f}}_{\mathbb{Q}_2})$. In this case, $(\alpha_{\rm est},\omega_{\rm est})=(2.2,0.75)$ and $(\alpha_{\rm opt},\omega_{\rm opt})=(1.8,0.775)$.} \label{tab:pair_FEM_toep}
		\end{center}
\end{table}


Figure \ref{fig:pair_Bspline} shows the magnitude of the spectral radius of the TGM iteration matrix for the matrix system  $C_n(\mathbf{f}^{(2,0)})$ computed over a range of admissible values for $\omega$ and $\alpha$. Precisely we consider $13$ equispaced values in $[0.7,1]$ for $\omega$ and $13$ equispaced values in $[1,1.6]$ for $\alpha$.

\begin{figure}[htb]
\centering
\includegraphics[width=0.70\textwidth]{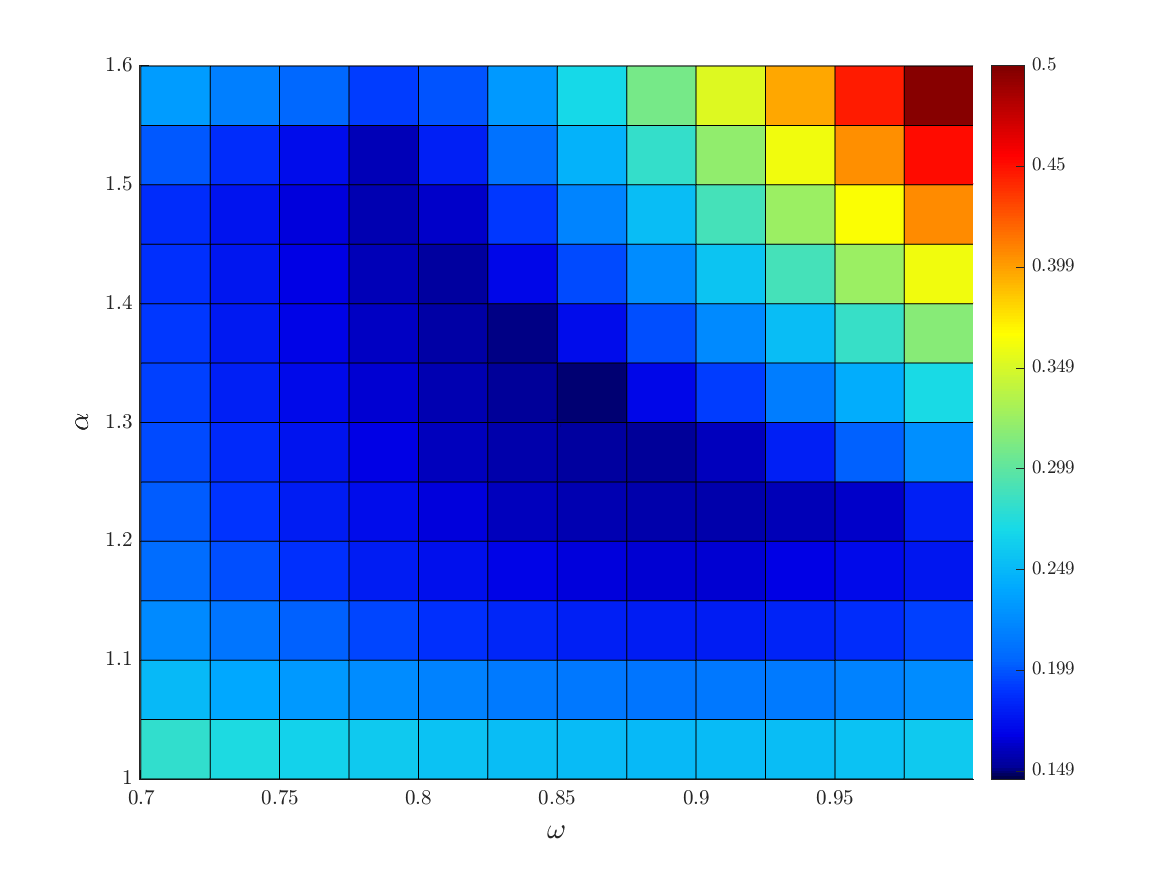}
\caption{Plot of the magnitude of the spectral radius of the TGM iteration matrix for the system   $A_{(2,0)}$  computed over equispaced values of the pair $(\alpha,\omega)$. }\label{fig:pair_Bspline}
\end{figure}

In Table \ref{tab:pair_B} we show the results in term of iterations with the estimated choices $\alpha_{\rm est}=1.3$ and $\omega_{\rm est}=0.85$ which provide $\rho(TGM)\approx 0.149$ in comparison with the case $\alpha=1$ (without over-relaxation)  and $\omega=0.85$ which provide $\rho(TGM)\approx 0.250$.

\begin{table}
\begin{center}
		\begin{tabular}{c|cc|cc}
			
		$N=2\cdot 2^t $	& \multicolumn{2}{c|}{{$(\alpha_{\rm est},\omega_{\rm est})=(1.3,0.85)$}} & \multicolumn{2}{c}{{$(\alpha=1,\omega=0.85)$}} \\
			\hline
			{$t$} & {TGM} & {V-Cycle} & {TGM} & {V-Cycle} \\
			\hline
			8 & 7 & 7 & 8 & 9  \\
			9 &  7& 7 &  8 & 10 \\
			10 &  7& 7 & 8 & 10\\
			11 &  7& 7 & 8 & 10  \\
			12&  7& 7 & 8 & 10  \\
			13 & 7 & 7  & 8& 9\\
		
	\end{tabular}
		\caption{ Two-grid and V-cycle iterations with and without the over-relaxation strategy for the matrix $C_n(\mathbf{f}^{(2,0)})$.}
		 \label{tab:pair_B}
		\end{center}
\end{table}

\begin{table}
\begin{center}
		\begin{tabular}{c|cc|cc|cc}			
		$N=2\cdot 2^t $	& \multicolumn{2}{c|}{{$(\alpha_{\rm est},\omega_{\rm est})$}} 	& \multicolumn{2}{c|}{{$(\alpha_{\rm opt},\omega_{\rm opt})$}} & \multicolumn{2}{c}{{$(\alpha=1,\omega=0.85)$}} \\
			\hline
			{$t$} & {TGM} & {V-Cycle} & {TGM} & {V-Cycle} & {TGM} & {V-Cycle} \\
			\hline
		    8 & 7 & 7 & 7 & 7 & 8 & 9 \\
			9 & 7 & 7 & 7 & 7 & 8 & 10 \\
			10 & 7 & 7 & 7 & 7 & 8 & 10\\
			11 & 7 & 7 & 7 & 7 & 8 & 10  \\
			12& 7 & 7 &7 & 7 & 8 & 10\\
                13& 7 & 7 &7 & 7 & 8 & 9\\	
	\end{tabular}
		\caption{ Two-grid and V-cycle iterations with and without the over-relaxation strategy for the matrix $T_n(\mathbf{f}^{(2,0)})$. In this case, $(\alpha_{\rm est},\omega_{\rm est})=(1.3,0.85)$ and $(\alpha_{\rm opt},\omega_{\rm opt})=(1.3,0.85)$.} \label{tab:pair_B_toep}
		\end{center}
\end{table}

\subsection{MGM as preconditioner in Krylov methods}\label{ssec:PPCG}

 We test the previously discussed block multigrid methods as preconditioners for solving the linear systems with the {PCG} method \cite{tatebe1993multigrid}. As test problems we consider the block Toeplitz matrices  $T_n(\mathbf{\textbf{f}}_{\mathbb{Q}_d})$ from Subsection \ref{ssec:FEM}, with $d=2,4,8$ and the matrix $T_n(\mathbf{f}^{(p,k)})$ of Subsection \ref{ssec:Splines} for  the pairs
$(p,k)$ equal to $(2,0)$, $(3,1)$ and $(3,0)$.  We use the built-in Matlab function \textit{pcg} and we set the zero vector as initial guess and the tolerance of the PCG method to $\epsilon=10^{-6}$.

Precisely, one iteration of the block symbol-based V-cycle method described in Subsection \ref{ssec:scalar_vs_block} is used as preconditioner for the PCG method. The resulting method is denoted by  $P_{\tiny{\mbox{block-symbol}}}$--PCG. We also test $P_{\tiny{\mbox{aggregate}}}$--PCG method where one iteration of the aggregated-based V-cycle method, tested in Subsections \ref{ssec:aggregation_experiments}--\ref{ssec:over_relaxation}, is used as preconditioner.  The V-cycle methods in both procedures use block Jacobi as pre and post smoothers at the finest level. However, we already observed that the coarser systems in matrix-valued multigrid approach maintain a block structure. Then, results of Table \ref{tab:simax_blockvsscalar_jacobi} suggest to exploit block Jacobi for the coarser levels. On the other hand, $P_{\tiny{\mbox{aggregate}}}$--PCG  uses the scalar Jacobi method at the coarser levels due to the scalar nature of the coarser linear systems in the aggregate-based approach. In all cases we select the smoothing and, when applicable, over-relaxation parameters that give the fastest convergence.

In Tables \ref{tab:PCG_tepl_Psimax_vs_Paggregation}  and \ref{tab:PCG_tepl_Psimax_vs_Paggregation_Bspline} we compare  the performances in terms of iterations needed for convergence and computational times. Both methods show a convergent and optimal behaviour. However, we highlight that the setup times $T_{\rm set}(s)$ for the block-symbol approach are bigger with respect to the ones of the aggregation-based strategy. This is expected since in the first case the setup involves block Jacobi smoothers and block grid transfer operators at the coarser levels. Consequently, even if the global number of iterations required by $P_{\tiny{\mbox{block-symbol}}}$--PCG is less than the one required by the $P_{\tiny{\mbox{aggregation}}}$--PCG, the aggregation-based is a preferable approach in combination with a Krylov subspace method.

\begin{footnotesize}
\begin{table}[htb]
	\begin{center}
		\begin{tabular}{c|c|ccc|ccc|ccc}
	&& \multicolumn{3}{c}{{$d=2$}}&\multicolumn{3}{c}{{$d=3$}}&\multicolumn{3}{c}{{$d=4$}}\\
			\hline
			
& $t $&  $T (s)$ & $T_{\rm set}(s)$ & {Iter} & $T (s)$ & $T_{\rm set}(s)$ & {Iter} & $T (s)$ &$T_{\rm set}(s)$  & {Iter} \\
\hline\hline
\multirow{5}{*}{${\tiny{\mbox{block-symbol}}}$} &		
						    12&0.0203 &0.0908 &6 &0.0291 &0.1430 &6&0.0318 &0.1997 &5\\
\cline{2-2}	&		13&0.0269 &0.3231 &6 &0.0430 &0.5268 &6&0.0467 &0.7395 &5\\
\cline{2-2}	&		14&0.0468 &1.2049 &6 &0.0771 &2.0112  &6&0.0864 &2.8552 &5\\
\cline{2-2}	&		15&0.0923 &4.5608 &6 &0.1457 &7.8515  &6&0.1644 &11.3633 &5\\
\cline{2-2}	&		16&0.1696 &17.9900 &6 &0.2924 &31.3316 &6&0.3260 &46.3115&5\\
\hline\hline	
\multirow{5}{*}{${\tiny{\mbox{aggregate}}}$} &
                            12&0.0139 &0.0658 &7 &0.0200 &0.1055 &8&0.0260 &0.1503  &8\\
\cline{2-2}	&		13&0.0265 &0.2429 &7 &0.0416 &0.3909 &8&0.0545 &0.5635  &8\\
\cline{2-2}	&		14&0.0372 &0.9009 &7 &0.0658 &1.4829 &8&0.0905 &2.2172   &8\\
\cline{2-2}	&		15&0.0729 &3.4918  &7&0.1228 &5.8319 &8&0.1658 &8.7442 &8\\
\cline{2-2}	&		16&0.1316 &13.5594&7&0.2361 &23.4031 &8&0.3282 &36.0278 &8\\
			\hline
		\end{tabular}
	\end{center}
	
	\caption{  {$P_{\tiny{\mbox{block-symbol}}}$--PCG and $P_{\tiny{\mbox{aggregate}}}$--PCG   number of iterations (IT),  CPU times $T(s)$ and setup times $T_{\rm set}(s)$ for the matrix $T_n(\mathbf{\textbf{f}}_{\mathbb{Q}_d})$.}\label{tab:PCG_tepl_Psimax_vs_Paggregation}
 }

\end{table}
\end{footnotesize}

\begin{footnotesize}
\begin{table}[htb]
	\begin{center}
		\begin{tabular}{c|c|ccc|ccc|ccc}
	&& \multicolumn{3}{c}{{$T_n(\mathbf{f}^{(2,0)})$}}&\multicolumn{3}{c}{{$T_n(\mathbf{f}^{(3,1)})$}}&\multicolumn{3}{c}{{$T_n(\mathbf{f}^{(3,0)})$}}\\
			\hline
			
& $t $&  $T (s)$ & $T_{\rm set}(s)$ & {Iter} & $T (s)$ & $T_{\rm set}(s)$ & {Iter} & $T (s)$ &$T_{\rm set}(s)$  & {Iter} \\
\hline\hline
\multirow{5}{*}{${\tiny{\mbox{block-symbol}}}$} &		
			    12&0.0178 &0.0936 &5 &0.0145 &0.0913 &5&0.0236&0.1432 &5\\
\cline{2-2}	&		13&0.0239 &0.3238 &5 &0.0196 &0.3064 &5&0.0352 &0.5247 &5\\
\cline{2-2}	&		14&0.0482 &1.2177 &5 &0.0411&1.1794 &5&0.0710 &2.0218 &5\\
\cline{2-2}	&		15&0.0735 &4.6862 &5 &0.0564&4.3287  &5&0.1238 &7.8767 &5\\
\cline{2-2}	&		16&0.1441 &18.2871 &5 &0.1028 &16.7949 &5&0.2379 &31.3311&5\\
\hline\hline	
\multirow{5}{*}{${\tiny{\mbox{aggregate}}}$} &
                    12&0.0105 &0.0708 &6 &0.0091 &0.0637 &6&0.0173 &0.1083  &7\\
\cline{2-2}	&		13&0.0254 &0.2461 &6 &0.0198 &0.2300 &6&0.0380 &0.3942  &7\\
\cline{2-2}	&		14&0.0316 &0.9123 &6 &0.0211 &0.7980 &6&0.0583 &1.4953   &7\\
\cline{2-2}	&		15&0.0598 &3.4895 &6&0.0313 &3.0448 &6&0.1112 &5.8689 &7\\
\cline{2-2}	&		16&0.1134 &13.5925&6&0.0583 &11.6552 &6&0.2033&23.4619 &7\\
			\hline
		\end{tabular}
	\end{center}
	
	\caption{  {$P_{\tiny{\mbox{block-symbol}}}$--PCG and $P_{\tiny{\mbox{aggregate}}}$--PCG   number of iterations (IT),  CPU times $T(s)$ and setup times $T_{\rm set}(s)$ for the matrix $T_n(\mathbf{f}^{(p,k)})$, 
 for the pairs
$(p,k)$ equal to $(2,0)$, $(3,1)$ and $(3,0)$.}\label{tab:PCG_tepl_Psimax_vs_Paggregation_Bspline}
 }

\end{table}
\end{footnotesize}

\section{Conclusions}\label{sec:conclusions}
In this paper, we investigated novel multigrid approaches for solving large block Toeplitz linear systems, emphasizing the analysis of symbols at coarse levels for V-cycle convergence of an aggregation-based approach and the efficiency of block Jacobi smoothers. Our findings simplify the theoretical analysis of symbol-based matrix-valued multigrid and make the choice of parameters computationally more feasible, reducing the calculations to scalar function evaluations. The efficiency of the aggregation-based MGM for block Toeplitz matrices is enhanced by an over-relaxation strategy, for which symbol computations are again the foundation for an a priori analysis. Through rigorous theoretical derivations and extensive numerical experiments, we provided a comparison among existing methods. To summarize, block Jacobi is preferable for larger blocks and the over-relaxed aggregation-based procedure is to be preferred to an approach which preserves the block structure at coarser levels only when used as a preconditioner.

While the efficacy of the aggregation-based procedure on block multilevel Toeplitz matrices was established in prior works \cite{An2024}, our future efforts will focus on exploring this approach in more complex contexts, specifically in multigrid methods for saddle point problems characterized by block (multilevel) Toeplitz submatrices. In particular, we aim to combine the approaches in \cite{saddle_point2,saddle_point1} with aggregation-base strategies.

 \section*{Acknowledgments}
The work of the second, third, and fourth authors is partly supported by “Gruppo Nazionale per il Calcolo Scientifico" (GNCS-INdAM).
Moreover, the work is supported by $\#$NEXTGENERATIONEU (NGEU) and funded by the Ministry of University and Research (MUR), National Recovery and Resilience Plan (NRRP), project MNESYS (PE0000006) – A Multiscale integrated approach to the study of the nervous system in health and disease (DN. 1553 11.10.2022).

\bibliographystyle{plain}

\begin{thebibliography}{10}

\bibitem{AD}
A.~Aric\`o and M.~Donatelli.
\newblock A {V}-cycle multigrid for multilevel matrix algebras: proof of
  optimality.
\newblock {\em Numer. Math.}, 105(4):511--547, 2007.

\bibitem{ADS}
A.~Aric\`o, M.~Donatelli, and S.~Serra-Capizzano.
\newblock {V}-cycle optimal convergence for certain (multilevel) structured
  linear systems.
\newblock {\em SIAM J. Matrix Anal. Appl.}, 26(1):186--214, 2004.

\bibitem{MR4389580}
M.~Bolten, M.~Donatelli, P.~Ferrari, and I.~Furci.
\newblock A symbol-based analysis for multigrid methods for block-circulant and
  block-{T}oeplitz systems.
\newblock {\em SIAM J. Matrix Anal. Appl.}, 43(1):405--438, 2022.

\bibitem{saddle_point2}
M.~Bolten, M.~Donatelli, P.~Ferrari, and I.~Furci.
\newblock Symbol based convergence analysis in block multigrid methods with
  applications for {S}tokes problems.
\newblock {\em Appl. Numer. Math.}, 193:109–130, 2023.

\bibitem{saddle_point1}
M.~Bolten, M.~Donatelli, P.~Ferrari, and I.~Furci.
\newblock Symbol based convergence analysis in multigrid methods for saddle
  point problems.
\newblock {\em Linear Algebra Appl.}, 671:67--108, 2023.

\bibitem{BDH}
M.~Bolten, M.~Donatelli, and T.~Huckle.
\newblock Analysis of smoothed aggregation multigrid methods based on
  {T}oeplitz matrices.
\newblock {\em Electron. Trans. Numer. Anal.}, 44:25--52, 2015.

\bibitem{MR1370108}
D.~Braess.
\newblock Towards algebraic multigrid for elliptic problems of second order.
\newblock {\em Computing}, 55(4):379--393, 1995.

\bibitem{MR2179900}
M.~Brezina, R.~Falgout, S.~MacLachlan, T.~Manteuffel, S.~McCormick, and
  J.~Ruge.
\newblock Adaptive smoothed aggregation {$(\alpha{\rm SA})$} multigrid.
\newblock {\em SIAM Rev.}, 47(2):317--346, 2005.

\bibitem{MR2369301}
Z.-H. Cao.
\newblock On the convergence of general stationary linear iterative methods for
  singular linear systems.
\newblock {\em SIAM J. Matrix Anal. Appl.}, 29(4):1382--1388, 2007.

\bibitem{CCS}
R.~H. Chan, Q.-S. Chang, and H.-W. Sun.
\newblock Multigrid method for ill-conditioned symmetric {T}oeplitz systems.
\newblock {\em SIAM J. Sci. Comput.}, 19(2):516--529, 1998.

\bibitem{An2024}
A.~Chengtao and S.~Yangfeng.
\newblock An aggregation-based two-grid method for multilevel block toeplitz
  linear systems.
\newblock {\em Journal of Scientific Computing}, 98(3), 2024.

\bibitem{MR4615364}
P.~D'Ambra, F.~Durastante, S.~Filippone, and L.~Zikatanov.
\newblock Automatic coarsening in algebraic multigrid utilizing quality
  measures for matching-based aggregations.
\newblock {\em Comput. Math. Appl.}, 144:290--305, 2023.

\bibitem{MR3264812}
P.~D'Ambra and P.~S. Vassilevski.
\newblock Adaptive {AMG} with coarsening based on compatible weighted matching.
\newblock {\em Comput. Vis. Sci.}, 16(2):59--76, 2013.

\bibitem{MR4284081}
M.~Donatelli, P.~Ferrari, I.~Furci, S.~Serra-Capizzano, and D.~Sesana.
\newblock Multigrid methods for block-{T}oeplitz linear systems: convergence
  analysis and applications.
\newblock {\em Numer. Linear Algebra Appl.}, 28(4):Paper No. e2356, 20, 2021.

\bibitem{MR2114296}
R.~D. Falgout and P.~S. Vassilevski.
\newblock On generalizing the algebraic multigrid framework.
\newblock {\em SIAM J. Numer. Anal.}, 42(4):1669--1693, 2004.

\bibitem{MR2150164}
R.~D. Falgout, P.~S. Vassilevski, and L.~T. Zikatanov.
\newblock On two-grid convergence estimates.
\newblock {\em Numer. Linear Algebra Appl.}, 12(5-6):471--494, 2005.

\bibitem{FRTBS}
P.~Ferrari, R.~I. Rahla, C.~Tablino-Possio, S.~Belhaj, and S.~Serra-Capizzano.
\newblock Multigrid for $\mathbb{Q}_{k}$ finite element matrices using a
  (block) {T}oeplitz symbol approach.
\newblock {\em Mathematics}, 8(5), 2020.

\bibitem{qp}
C.~Garoni, S.~Serra-Capizzano, and D.~Sesana.
\newblock Spectral analysis and spectral symbol of {$d$}-variate {$\Bbb{Q}_p$}
  {L}agrangian {FEM} stiffness matrices.
\newblock {\em SIAM J. Matrix Anal. Appl.}, 36(3):1100--1128, 2015.

\bibitem{MR4019290}
C.~Garoni, H.~Speleers, S.-E. Ekstr\"{o}m, A.~Reali, S.~Serra-Capizzano, and
  T.~J.~R. Hughes.
\newblock Symbol-based analysis of finite element and isogeometric {B}-spline
  discretizations of eigenvalue problems: exposition and review.
\newblock {\em Arch. Comput. Methods Eng.}, 26(5):1639--1690, 2019.

\bibitem{HS2}
T.~Huckle and J.~Staudacher.
\newblock Multigrid methods for block {T}oeplitz matrices with small size
  blocks.
\newblock {\em BIT}, 46(1):61--83, 2006.

\bibitem{huckle}
T.~K. Huckle.
\newblock Compact fourier analysis for designing multigrid methods.
\newblock {\em SIAM Journal on Scientific Computing}, 31(1):644--666, 2008.

\bibitem{MR2385900}
A.~C. Muresan and Y.~Notay.
\newblock Analysis of aggregation-based multigrid.
\newblock {\em SIAM J. Sci. Comput.}, 30(2):1082--1103, 2008.

\bibitem{RStub}
J.~W. Ruge and K.~St\"{u}ben.
\newblock Algebraic multigrid.
\newblock In {\em Multigrid methods}, volume~3 of {\em Frontiers Appl. Math.},
  pages 73--130. SIAM, Philadelphia, PA, 1987.

\bibitem{MR1820878}
K.~St\"{u}ben.
\newblock A review of algebraic multigrid.
\newblock volume 128, pages 281--309. 2001.
\newblock Numerical analysis 2000, Vol. VII, Partial differential equations.

\bibitem{tatebe1993multigrid}
O.~Tatebe.
\newblock The multigrid preconditioned conjugate gradient method.
\newblock In {\em NASA. Langley Research Center, The Sixth Copper Mountain
  Conference on Multigrid Methods, Part 2}, pages 621--634, 1993.

\bibitem{Trot}
U.~Trottenberg, C.~W. Oosterlee, and A.~Sch\"{u}ller.
\newblock {\em Multigrid}.
\newblock Academic Press, Inc., San Diego, CA, 2001. With contributions by A.
  Brandt, P. Oswald and K. St\"{u}ben.

\end{thebibliography}

\end{document}